\documentclass[12pt]{amsart}
\usepackage{pifont}
\usepackage{txfonts}
\usepackage{}
\usepackage{amsfonts}
\usepackage{graphicx}
\usepackage{epstopdf}
\usepackage{amsmath}
\usepackage{amsthm}
\usepackage{latexsym,bm}
\usepackage{amssymb}
\usepackage{esint}
\usepackage{enumitem}
\usepackage{indentfirst}
\usepackage{amscd}
\newtheorem{thm}{Theorem}[section]
\newtheorem{cor}[thm]{Corollary}
\newtheorem{lem}[thm]{Lemma}
\newtheorem{prop}[thm]{Proposition}
\newtheorem{defn}[thm]{Definition}
\numberwithin{equation}{section}
\numberwithin{Remark}{section}

\begin{document}

\title{Gap phenomena and curvature estimates for Conformally Compact Einstein Manifolds}

\author{Gang Li$^\dag$, Jie Qing$^*$ and Yuguang Shi$^\ddag$}

\begin{abstract} In this paper we first use the result in \cite{CN} to remove the assumption of the $L^2$ boundedness
of Weyl curvature in the gap theorem in \cite{CQY1} and then obtain a gap theorem for a class of conformally compact Einstein
manifolds with very large renormalized volume. We also uses the blow-up method to derive curvature estimates for conformally
compact Einstein manifolds with large renormalized volume. The second part of this paper is on conformally compact Einstein
manifolds with conformal infinities of large Yamabe constants. Based on the idea in \cite{DJ} we manage to give the complete
proof of the relative volume inequality \eqref{volume-com} on conformally compact Einstein manifolds. Therefore we obtain the complete
proof of the rigidity theorem for conformally compact Einstein manifolds in general dimensions with no spin structure assumption
(cf. \cite{Q,DJ}) as well as the new curvature pinch estimates for conformally compact Einstein manifolds with conformal infinities of very large
Yamabe constant. We also derive the curvature estimates for conformally compact Einstein manifolds with conformal infinities of large
Yamabe constant.
\end{abstract}

\renewcommand{\subjclassname}{\textup{2000} Mathematics Subject Classification}

\address{Gang Li, Beijing International Center for Mathematical Research, Peking University, Beijing, China}
\email{runxing3@gmail.com}

\address{Jie Qing, Department of Mathematics, UC Santa Cruz, California, USA}
\email{qing@ucsc.edu}

\address{Yuguang Shi, Key Laboratory of Pure and Applied
mathematics, School of Mathematics Science, Peking University,
Beijing, 100871, P.R. China.}
\email{ygshi@math.pku.edu.cn}

\thanks{$^\dag$ Research supported by China Postdoctoral Science Foundation Grant 2014M550540.}

\thanks{$^*$ Research supported by NSF grant DMS-1303543.}

\thanks{$^\ddag$ Research supported by   NSF grant of China 10990013.}

\maketitle


\section{Introduction}

The study of conformally compact Einstein manifolds is fundamental in establishing mathematical
theory of the so-called AdS/CFT correspondence proposed in the theory of quantum gravity in theoretic physics.
It is well understood that there is the rigidity phenomenon for conformally compact Einstein manifolds, that is,
a conformally compact Einstein manifold whose conformal infinity is the conformal round sphere has to be the hyperbolic
space \cite{AD,Q}. On the other hand, in \cite{GL,Bq}, it was shown that for each conformal
sphere that is sufficiently close to the conformal round sphere there exists a conformally compact Einstein metric
on the ball whose conformal infinity is the given conformal sphere. Those conformally compact Einstein metrics
constructed in \cite{GL,Bq} are automatically close to the hyperbolic space in some appropriate way. In an attempt to understand
if those conformally compact Einstein metrics given in \cite{GL,Bq} are the unique ones, in this paper,  we describe some gap phenomena
in terms of renormalized volumes and derive curvature estimates when either the renormalized volume is close to that
of hyperbolic space or the Yamabe constant of the conformal infinity is close to that of conformal round sphere. \\

The gap theorem in this paper for renormalized volumes of conformally compact Einstein 4-manifolds grows out of the gap theorem
in \cite{CQY1} for closed Bach flat 4-manifolds. As a consequence of recent remarkable work of Cheeger and
Naber \cite{CN} we first remove the dependence of the $L^2$ of the Weyl curvature from the gap theorem in \cite{CQY1}
and obtain the following:

\begin{thm}\label{conformal-gap} There exists a positive small number $\epsilon$ such that a closed Bach flat 4-manifold $(M, \ g)$ of positive Yamabe type
has to be conformally equivalent to the round 4-sphere, if
$$
\int_M \sigma_2(A[g]) d V_g \geq (1-\epsilon)16\pi^2,
$$
where $\sigma_2(A[g])$ is the second symmetric function of the eigenvalues of the
Schouten curvature tensor $A[g] = \frac 1{n-2}(\text{Ric} - \frac R{2(n-1)}g)$ of the metric $g$.
\end{thm}

Our proof of Theorem \ref{conformal-gap} is slightly different from that in \cite{CQY1}. Our approach instead replies on the control of the Yamabe constant.
The Yamabe constant is defined as follows:

\begin{equation}\label{yamabe}
Y(M, \ g) = \inf \{ \frac {\int_M (\frac {4(n-1)}{n-2}|\nabla u|^2 + R[g]u^2)dv[g]}{(\int_M u^{\frac {2n}{n-2}}dv[g])^\frac {n-2}n}: u \in C_c^\infty(M)\backslash\{0\} \},
\end{equation}
where $C^\infty_c(M)$ is the space of smooth functions with compact support in $M$.
It is easily seen that $Y(M, \ g)$ is a conformal invariant when $(M, \ g)$ is compact with no boundary, in which case we denote it by $Y(M, \ [g])$ instead. On the other hand,
on a complete non-compact manifold, $Y(M, \ g)$ is rather a local Yamabe constant. For instance
$$
Y(\mathbb{R}^n, \ g_E) = Y(\mathbb{H}^n, \ g_{\mathbb{H}}) = Y(\mathbb{S}^n, \ [g_{\mathbb{S}}]).
$$
The Yamabe constant is as convenient to use as the Euclidean volume growth bound when rescaling metrics.  Particularly we observe the following fact:

\begin{lem}\label{lemma1.2} Suppose that $(X^4, \ g)$ is a complete non-compact Ricci flat manifold. And suppose that
$$
Y(M, g) > 2^{-\frac 12} Y(\mathbb{R}^4, g_{E}).
$$
Then $(X^4, \ g)$ is isometric to the Euclidean 4-space $(\mathbb{R}^4, g_E)$.
\end{lem}

Lemma \ref{lemma1.2} is the special case of Lemma \ref{trivialgroup} when $n=4$. The proof of this fact is rather straightforward using the end analysis based on \cite{And1,BKN} and the simple consequence Lemma \ref{improved} from
the new remarkable work in \cite{CN} in dimension 4. \\

To see the control of Yamabe constant from the integral of $\sigma_2(A)$, with the Yamabe metric $g_Y$ on a compact manifold, we have
\begin{equation}\label{sigma-2}
\int_M (\sigma_2(A)dv)[g_Y] = \int_M ((\frac 1{24}R^2 - \frac 12 |\overset{\circ}{\text{Ric}}|^2)dv)[g_Y] \leq \frac 1{24}(Y(M, [g]))^2.
\end{equation}
It is known from \cite{Graham} that, on a conformally compact Einstein 4-manifold $(X^4, g^+)$ with the conformal infinity $(\partial X, [\hat g])$,
\begin{equation}\label{volume-exp}
\text{Vol}(\{x > \epsilon\}) = \frac 13 \text{Vol}(\partial X, \hat g) \epsilon^{-3} - \frac 18 \int_{\partial X} (Rdv)[\hat g] \epsilon^{-1} + V(X^4, \ g^+) + o(1)
\end{equation}
where $R[\hat g]$ is the scalar curvature of the metric $\hat g$ and $x$ is the geodesic defining function associated with a representative $\hat g$ of
the conformal infinity. It turns out that $V(X^4, \ g^+)$ in \eqref{volume-exp} is independent of representatives and is called the renormalized volume \cite{HS,Graham}. The expansion \eqref{volume-exp} uses the expansion of the Einstein metric $g^+$ given in \cite{FG}
\begin{equation}\label{metric-exp}
g^+ = x^{-2}(dx^2 + g_x) = x^{-2} ( dx^2 + \hat g + g^{(2)}x^2 + g^{(3)}x^3 + o(x^3))
\end{equation}
where $g^{(2)}$ is a curvature tensor of $\hat g$ while $g^{(3)}$ is non-local.
Moreover, in \cite{And2} (see also \cite{CQY2}), it is shown that
\begin{equation}\label{Gauss-Bonnet}
8\pi^2 \chi(X) = \frac{1}{4} \int_X (|W|^2 d v)[g^+] + 6 V(X^4, \ g^+)
\end{equation}
for a conformally compact Einstein 4-manifold $(X^4, \ g^+)$. Consequently, one knows
\begin{equation}\label{v-sigma-2}
\int_{X^4} (\sigma_2(A)dv)[\bar g] = 6 V(X^4, \ g^+),
\end{equation}
where $\bar g = x^2 g^+$ is a conformal compacitification. Therefore, given a conformally compact Einstein manifold, one may control the Yamabe constant
for compactifications from the renormalized volume. Following the idea in \cite{CQY3}, one may consider the doubling of the compactified manifold from
a conformally compact Einstein manifold and obtain the corresponding gap theorem.

\begin{thm}\label{CGT} There exists a small positive number $\epsilon$ such that a conformally compact Einstein 4-manifold $(X^4, \ g^+)$
with the conformal infinity of positive Yamabe type has to be isometric to the hyperbolic space, if its renormalized volume satisfies
\begin{equation}\label{volume-gap}
V(X^4, \ g^+) \geq (1 - \epsilon)\frac {4\pi^2}3 = (1-\epsilon)V(\mathbb{H}^4, \ g_{\mathbb{H}})
\end{equation}
and the non-local term $g^{(3)}$ in \eqref{metric-exp} for $g^+$ vanishes.
\end{thm}

Theorem \ref{CGT} indicates that the hyperbolic space is the only ``critical point'' of the renormalized volume among all conformally compact Einstein manifolds that
satisfy \eqref{volume-gap}, since the Euler-Lagrangian equation for the renormalized volume is $g^{(3)} = 0$ by the calculation made in \cite{And2} . On the other hand,
Theorem \ref{CGT} clearly does not hold if one drops the assumption $g^{(3)} = 0$. In fact the conformally compact Einstein metrics constructed in
\cite{GL,Bq} provide plenty of examples of conformally compact Einstein metrics that satisfy \eqref{volume-gap} for arbitrarily small $\epsilon$
and with conformal infinities of positive Yamabe type. Nevertheless, one obtains the following curvature estimate for conformally compact Einstein metrics when
the renormalized volume is large enough.

\begin{thm} \label{curvature-bound} For any $0 < \epsilon < \frac 12$, there is a positive constant $C$ such that, on a conformally compact Einstein 4-manifold $(X, \ g^+)$,
$$
\|\text{Rm}\| [g^+]\leq C,
$$
where $\text{Rm}[g^+]$ is the Riemann curvature tensor of the metric $g^+$, provided that the conformal infinity is of positive Yamabe type and
\eqref{volume-gap} holds.
\end{thm}

Next we focus our attention to those conformally compact Einstein manifolds whose conformal infinities have large Yamabe constants. It is a very original idea in \cite{DJ} to
use the Yamabe constant of the conformal infinity to control the relative volume growth of geodesic balls in conformally compact Einstein manifolds.
We recognize the important contribution of \cite{DJ}, but are compelled to present a complete and correct proof of the following:

\begin{thm}\label{main-dutta} Suppose that $(X^n, \ g^+)$ is an AH manifold of $C^3$ regularity and with the conformal infinity of positive Yamabe type. Let
$p\in X^n$ be a fixed point. Assume
 \begin{equation}\label{alter-einstein}
 \text{Ric}[g^+]\geq -(n-1)g^+ \text{ and } \ R[g^+]+n(n-1)= o(e^{-2t})
 \end{equation}
 for the distance function $t$ from $p$. Then
\begin{equation}\label{volume-com}
(\frac{Y(\partial X, [\hat{g}])}{Y(\mathbb{S}^{n-1}, [g_{\mathbb{S}}])})^{\frac{n-1}{2}}\leq \frac{Vol(\partial B_{g^+}(p, t))}
{Vol(\partial B_{g_{\mathbb{H}}}(0, t)} \leq \frac{Vol(B_{g^+}(p, t))}
{Vol(B_{g_{\mathbb{H}}}(0, t)}\leq 1,
\end{equation}
where $B_{g^+}(p, t)$ and $B_{g_{\mathbb{H}}}(0, t)$ are geodesic balls.
\end{thm}

The $C^3$ regularity is used to construct the geodesic defining function $x$ and $C^2$ conformal compactification $\bar g = x^2g^+$, for each given representative $\hat g$,
To start the proof we first need to clear a technical issue.

\begin{lem}\label{lem:cur-decay} Suppose that $(X^n, \ g^+)$ is AH of $C^2$ regularity and that $x$ is a defining function. Assume
\begin{equation}\label{alter-einstein-1}
 \text{Ric}[g^+]\geq -(n-1)g^+ \text{ and } \ R[g^+]+n(n-1)= o(x)
 \end{equation}
Then there is a constant $C_0$ such that
\begin{equation}\label{cur-decay}
|K[g^+]+1|\leq C_0x^2
\end{equation}
for any sectional curvature $K$.
\end{lem}

It takes substantial arguments to finish the proof of Theorem \ref{main-dutta} based on the idea presented in \cite{DJ}. The first issue is that \eqref{gap-1} (cf. (3.17) in \cite{DJ})
may not be available as claimed in \cite{DJ}, since the distance function $t$ is only Liptschitz in general. We devote Section \ref{Sect:normal-cut-point} to solve this issue by
the careful study of cut loci based on \cite{Ozols, Tanaka}. The second issue is that the estimate \eqref{2.3-dutta} (cf. (2.3) in \cite{DJ}) is not known to hold without assuming
the convexity of the geodesic spheres (cf. \cite{ST, HQS}).  We devote Section \ref{Sect:curvature-estimate} to derive the total scalar curvature estimate \eqref{total-cur-need} without
\eqref{2.3-dutta}. Our argument in Section \ref{Sect:curvature-estimate} uses more delicate analysis of the
Riccati equations on AH manifolds and volume estimates along geodesics where the mean curvature of the geodesic sphere is small.
\\

As argued in \cite{DJ},  Theorem \ref{main-dutta} implies the rigidity of conformally compact Einstein manifolds for any dimension, which are the cases where the conformal infinity is
exactly the round spheres (cf. \cite{M-O,AD,leung,Q,Wang,CH,BMQ,ST,DJ}).

\begin{thm} \label{rigidity} Suppose that $(X^n, \ g^+)$ is AH of $C^3$ regularity and that \eqref{alter-einstein} holds. Then $(X^n, g^+)$ is
isometric to the hyperbolic space $(\mathbb{H}^n, \ g_{\mathbb{H}})$, provided that the conformal infinity
$(\partial X, \ [\hat g])$ is the round conformal sphere.
\end{thm}

Notice that, when the conformal infinity is the round sphere, a conformally compact Einstein metric $g^+$ is always smooth according to \cite{CDLS}, provided that
it is at least of $C^2$ regularity.
Therefore Theorem \ref{rigidity} does cover the most general rigidity theorem for conformally compact Einstein manifolds whose conformal infinity is the round sphere in any dimension.
In this paper we also deduce from the relative volume growth estimates \eqref{volume-com} the following interesting curvature pinch estimates.

\begin{thm}\label{EHBoundary} For any $\epsilon >0$, there exists $\delta > 0$, for any conformally compact Einstein manifold $(X^{n}, g^+)$ ($n\geq 4$),
one gets
\begin{equation}\label{close-to-hyper}
|K [g^+] + 1| \leq \epsilon,
\end{equation}
for all sectional curvature $K$ of $g^+$, provided that
$$
Y(\partial X, [\hat{g}]) \geq (1-\delta) Y(S^{n-1}, [g_{\mathbb{S}}]).
$$
Particularly, any conformally compact Einstein manifold with its conformal infinity of Yamabe constant sufficiently close to that of the round sphere is
necessarily negatively curved.
\end{thm}

This result is even more interesting because it gives the curvature pinch estimate which only relies on the Yamabe constant of the conformal infinity. Particularly,
from Theorem \ref{EHBoundary} we now know that any conformally compact Einstein manifold whose conformal infinity is prescribed as a conformal
sphere that is sufficiently close to the round conformal sphere is negatively curved, which was only known to be true for those conformally compact Einstein
manifolds constructed in \cite{GL, Bq}.
\\

As a consequence of the proof of Theorem \ref{EHBoundary} we also get:

\begin{cor}\label{CCEflat} For $L>0$, $\tau > \frac 12$ and $n\geq 4$, there is a number $C$ such that, on any conformally compact Einstein manifold $(X^n, \ g^+)$,
$$
|W|[g^+] \leq C
$$
provided that $$Y(\partial X, [\hat g]) \geq \tau^\frac 2{n-1} Y(\mathbb{S}^n, [g_{\mathbb{S}}])$$ and $$\int_{X^n}(|W|^\frac n2dv)[g^+]\leq L$$ when $n\geq 5$.
\end{cor}

The organization of this paper is as follows: in Section \ref{Sect:preliminary}, we will introduce some basics about AH manifolds and conformally compact Einstein
manifolds. Particularly we will prove Lemma \ref{lem:cur-decay}.
In Section \ref{Sect:large-volume} we use the renormalized volume to control the Yamabe constant of the conformally compact Einstein 4-manifolds and
prove Theorem \ref{CGT} of the gap phenomenon and Theorem \ref{curvature-bound} of curvature estimates. In Section \ref{Sect:large-yamabe} We will first sketch a
proof of Theorem \ref{main-dutta} based on the idea in \cite{DJ}. We will identify the incompleteness and incorrectness of the arguments in \cite{DJ} in the sketch of the
proof of Theorem \ref{main-dutta}.
Then we will use Theorem \ref{main-dutta} to obtain Theorem \ref{EHBoundary} and Corollary \ref{CCEflat}. In Section \ref{Sect:normal-cut-point}
and Section \ref{Sect:curvature-estimate} we will resolve the gaps identified in Section \ref{Sect:large-yamabe}
and present the complete and correct proof of Theorem \ref{main-dutta} with details.

\vskip 0.2in\noindent
{\bf Acknowledgment} The authors are very grateful to Beijing International Center for Mathematical Research (BICMR) for the hospitality. Most of research in this paper
was conducted during the summer when the authors visit BICMR. The first author would like to thank Professor Matthew Gursky for his interest of the problem and encouragement


\section{Preliminaries}\label{Sect:preliminary}

Let us recall some basics about AH manifolds and conformally compact Einstein manifolds. First we use the following definition for conformally compact Einstein manifolds.

\begin{defn}\label{def-AH} Suppose that $X^{n}$ is the interior of a smooth compact manifold $\overline{X^{n}}$ with boundary $\partial X^{n-1}$.
A Riemannian metric $g^+$ on $X^{n}$ is said to be conformally compact of $C^{k, \alpha}$ regularity if, for a smooth defining function $x$ for
the boundary $\partial X^{n-1}$ in $\overline{X^{n}}$, $\bar g = x^2 g^+$ can be extended to a $C^{k, \alpha}$ Riemannian metric on $\overline{X^{n}}$.
If, in addition, $|dx|^2_{x^2g^+}|_{x=0}=1$, then we say $(X^{n}, \ g^+)$ is
asymptotically hyperbolic (AH in short) of $C^{k, \alpha}$ regularity. And if, in addition, $g^+$ is at least of $C^2$ regularity and Einstein, that is,
\begin{equation}\label{einstein}
\text{Ric}[g^+] = - (n-1) g^+,
\end{equation}
then we say
$(X^{n}, \ g^+)$ is a conformally compact Einstein manifold.
\end{defn}

A smooth defining function $x$ for the boundary $\partial X$ in a smooth manifold $\overline{X^{n}}$ is a smooth nonnegative function from $\overline{X^{n}}$ such that
\begin{itemize}
\item $x > 0$ in the interior $X^{n}$;
\item $x = 0$ on the boundary $\partial X^{n-1}$;
\item $dx \neq 0$ on the boundary $\partial X^{n-1}$.
\end{itemize}
The compactification $\bar g$ induces a metric $\hat g$ on the boundary $\partial X^{n-1}$ and changes conformally when the defining function $x$ varies. Hence a conformally compact
metric $g^+$ always induces a conformal structure $[\hat g]$ on the boundary $\partial X^{n-1}$. The conformal manifold $(\partial X^{n-1}, \  [\hat g])$ is
called the conformal infinity of the conformally compact manifold $(X^{n}, \ g^+)$. \\

Before we recall basic properties of conformally compact Einstein manifolds, we give a proof of Lemma \ref{lem:cur-decay} based on the proof of Lemma 3.1 in \cite{CDLS}.
\\

\noindent{\it Proof of Lemma \ref{lem:cur-decay}} \quad The first step completely follows the proof of Lemma 3.1 in \cite{CDLS} and concludes that, there is a coordinate at infinity
(up to a $C^3$ collar diffeomorphism in the language in \cite{CDLS}) such that, for a defining function $x$,
\begin{equation}\label{good-cor}
\bar g = x^2 g^+ = dx^2 + \hat g + g^{(1)}x + O(x^2)\in C^2(\overline{X^n})
\end{equation}
for some symmetric 2-tensor $g^{(1)}$ and a representative $\hat g = x^2g^+|_{T\partial X}$ on $\partial X$.
\\

We then calculate the transforms of Riemann curvature, Ricci curvature and scalar curvature based on $g^+ = x^{-2}\bar g$.
$$
R_{ijkl}[g^+] = -(g^+_{ik}g^+_{jl} - g^+_{jk}g^+_{il}) + x^{-1}(g^+_{ik}\nabla^{\bar g}_j\nabla^{\bar g}_l x + g^+_{jl}\nabla^{\bar g}_i\nabla^{\bar g}_k x -
g^+_{il}\nabla^{\bar g}_j\nabla^{\bar g}_k x - g^+_{jk}\nabla^{\bar g}_i\nabla^{\bar g}_l x) + O(x^2),
$$
$$
R_{ik}[g^+] = -(n-1)g^+ + x^{-1}((n-2)\nabla^{\bar g}_i\nabla^{\bar g}_k x + \Delta[\bar g] x \bar g_{ik}) + O(x^2),
$$
and
$$
R[g^+] = -n(n-1) + 2x(n-1)\Delta[\bar g] x + O(x^2).
$$
In the same time we can calculate
$$
\nabla^{\bar g}_i\nabla^{\bar g}_k x = \frac 12 \partial_x \bar g_{ik} + O(x).
$$
Therefore the condition \eqref{alter-einstein-1} are translated to
$$
(n-2)\partial_x \bar g_{ik} + \bar g^{jl}\partial_x \bar g_{jl} \bar g^{ik} \geq 0 \text{  and  } \bar g^{jl} \partial_x g_{jl} = 0 \text{ at $x=0$},
$$
which implies $g^{(1)}=0$ and thus \eqref{cur-decay}.
\qed
\\

The fundamental properties of conformally compact Einstein 4-manifolds that are useful to us are summarized in the following:

\begin{lem}
( \cite{FG,Graham,CDLS, Mazzeo1}) Let $(X^4, \ g^+)$ be a conformally compact Einstein manifold and $x$ be the geodesic defining function
associated with a representative $\hat{g}$ of the conformal infinity $(\partial X^3, [\hat{g}])$. In a neighborhood of the infinity
\begin{align}\label{canonical-form}
g=x^{-2}\tilde{g} = x^{-2}(dx^2+g_x)
\end{align}
with the expansion
\begin{align}\label{metric-expansion}
g_x=\hat{g} + g^{(2)}x^2 + g^{(3)}x^3 +\sum_{k = 4}^{m} g^{(k)}x^k + o(x^m)
\end{align}
for any $m\geq 4$, where $g^{(k)}$ is a symmetric $(0, 2)$ tensor on $\partial X$ for all $k$ and $g^{(3)}$ is the so-called non-local term. Moreover,
the expansion \eqref{metric-expansion} only has even power terms, when $g^{(3)}$ vanishes.
\end{lem}

As a consequence of the expansion \eqref{metric-expansion} one gets the following volume expansion:

\begin{lem} (\cite{HS,Graham}) Let $(X^4, \ g^+)$ be a conformally compact Einstein manifold and $x$ be the geodesic defining function
associated with a representative $\hat{g}$ of the conformal infinity $(\partial X^3, [\hat{g}])$.  One has
\begin{align*}
\text{Vol}(\{x > \epsilon\}) = \frac 13 \text{Vol}(\partial X, \hat g) \epsilon^{-3} - \frac 18 \int_{\partial X} (Rdv)[\hat g] \epsilon^{-1} + V(X^4, \ g^+) + o(1).
\end{align*}
More importantly $V(X^4, \ g^+)$ is independent of the choice of representative $\hat g$ and is called the renormalized volume.
\end{lem}

To appreciate the global invariant $V(X^4, \ g^+)$ of a conformally compact Einstein 4-manifold $(X^4, \ g^+)$ we recall the Gauss-Bonnet formula observed
in \cite{And2,CQY3}.

\begin{lem} Let $(X^4, \ g^+)$ be a conformally compact Einstein 4-manifold. Then
\begin{equation}\label{Gauss-Bonnet-2}
8\pi^2 \chi(X) = \frac{1}{4} \int_X (|W|^2 d v)[g^+] + 6 V(X^4, \ g^+)
\end{equation}
\end{lem}

Comparing the Gauss-Bonnet formula \eqref{Gauss-Bonnet-2} with the Gauss-Bonnet formula for compact 4-manifold $(X, \bar g)$ with totally geodesic boundary:
\begin{equation}\label{Gauss-Bonnet-b}
8\pi^2 \chi(X) = \frac{1}{4} \int_X (|W|^2 d v)[\bar g] + \int_X (\sigma_2(A)dv)[\bar g],
\end{equation}
we arrive at
\begin{equation}\label{sigma-2-gauss}
\int_X (\sigma_2(A)dv)[\bar g] = 6 V(X^4, \ g^4),
\end{equation}
for any compactification $\bar g = x^2 g^+$.
\\

To discuss the conformal gap theorem in \cite{CQY1} we recall the definition of Bach curvature tensor and $\epsilon$-regularity theorem for Bach flat Yamabe metrics.
On 4-manifolds the Bach curvature tensor is a symmetric 2-tensor
\begin{equation}\label{bach}
B_{ij} =  W_{kijl,}^{\quad lk} + \frac 12 R^{kl}W_{kijl}.
\end{equation}
The $\epsilon$-regularity theorem has been established in \cite{Tian-Viaclovsky} as follows:

\begin{lem}\label{TV-curvature} (\cite{Tian-Viaclovsky}) Suppose that $(M^4, g)$ is a
Bach flat 4-manifold with Yamabe constant $Y >0$ and $g$ is a
Yamabe metric. Then there exist positive numbers $\tau_k$ and
$C_k$ depending on $Y$ such that, for each geodesic ball
$B_{2r}(p)$ centered at $p\in M$, if
$$
\int_{B_{2r}(p)} |Rm|^2 dv \leq \tau_k,
$$
then
\begin{equation}\label{tv-curavture}
\sup_{B_r(p)} |\nabla^k Rm| \leq  \frac {C_k}{r^{2+k}}
(\int_{B_{2r}(p)} |Rm|^2 dv)^\frac 12.
\end{equation}
\end{lem}


\section{Conformally compact Einstein manifolds with large renormalized volumes}\label{Sect:large-volume}

In this section we use Cheeger and Naber's result in \cite{CN} in dimension $4$ to drop the $L^2$ boundedness condition in the conformal gap theorem in \cite{CQY1}
and prove a version of conformal gap theorem of renormalized volumes on conformally compact Einstein manifolds. First we state Cheeger and Naber's result.

\begin{thm}\label{Weylint4}(Theorem 1.5 \cite{CN}) There exists $C = C(v)$ such that, if $M^4$ satisfies $|Ric_{M^4} | \leq 3$ and $Vol(B_1(p)) > v > 0$, then
\begin{equation}\label{cheeger-naber}
\fint_{B_1(p)}|Rm|^2dV\leq C,
\end{equation}
where $B_1(p)$ is a geodesic ball on $(M^4, g)$.
\end{thm}

We then recall the following theorem on the end analysis in \cite{And1,BKN}.

\begin{thm}\label{thmasympflat} (\cite{And1,BKN})
Let $(M^n, g)$ ($n\geq 4$) be a complete noncompact Ricci flat Riemannian manifold satisfying that
\begin{align}
&\label{volume-growth}Vol(B_r)\geq C r^n\,\,\text{for all}\,\,r>0,\\
&\label{assump2}\int_M (|Rm|^{\frac{n}{2}} d V)[g] <\infty.
\end{align}
Then given $q\in M$, there is an $R_0>0$ such that
\begin{itemize}
\item $M \setminus B_{R_0}(q)$ is diffeomorphic to the cone $(R_0, \infty)\times S^{n-1}/\Gamma$;
\item $g  = g_F + O(r^{-2})$, where $g_F$ is the flat metric on the cone.
\end{itemize}
Moreover, if $M$ is simply connected at infinity, then $(M^n, \ g)$ is isometric to the Euclidean space.
\end{thm}

It turns out, as a straightforward consequence of Theorem \ref{Weylint4} of Cheeger and Naber,
one can drop the assumption \eqref{assump2} in Theorem \ref{thmasympflat} in dimension 4.

\begin{lem}\label{improved}
Let $(M^4, g)$ be a complete noncompact Ricci flat Riemannian 4-manifold satisfying the Euclidean volume growth assumption \eqref{volume-growth}.
Then, given $q\in M$, there is an $R_0>0$ such that
\begin{itemize}
\item $M \setminus B_{R_0}(q)$ is diffeomorphic to the cone $(R_0, \infty)\times S^3/\Gamma$;
\item $g  = g_F + O(r^{-2})$, where $g_F$ is the flat metric on the cone.
\end{itemize}
Moreover, if $M$ is simply connected at infinity, then $(M^4, \ g)$ is isometric to the Euclidean 4-space.
\end{lem}

\begin{proof} In the light of Theorem \ref{thmasympflat} (\cite{And1,BKN}) it suffices to show that, there is a constant $C$ such that
$$
\int_{B_R} (|Rm|^2dv)[g] \leq C
$$
for any $R>0$. For any fixed $R>0$, we consider the metric $g_R = R^{-2}g$. It is easily verified that Theorem \ref{Weylint4} is applicable to $g_R$. Hence
the proof is complete.
\end{proof}

Theorem \ref{conformal-gap} is an improved version of the conformal gap theorem in \cite{CQY1} based on the above Lemma \ref{improved}.  But we will present a complete proof of
Theorem \ref{conformal-gap} that is slightly different from that in \cite{CQY1}. Our approach replies on the control of the Yamabe constant via \eqref{sigma-2}
and the following observation that describes the influence of the end structure from the Yamabe constant similar to that of the lower bound of the Euclidean volume growth.

\begin{lem}\label{trivialgroup}
Let $(M^n, \ g)$ ($n\geq 4$) be a complete noncompact Ricci flat Riemannian manifold. Assume \eqref{assump2} when $n \geq 5$. Then $(M^n, \ g)$ is isometric to the Euclidean
space $\mathbb{R}^n$, provided that
\begin{align}\label{inequlocSob}
Y(M, \ g) \geq \tau Y(\mathbb{R}^n, g_E)
\end{align}
for some $\tau > 2 ^{-\frac{2}{n}}$.
\end{lem}
\begin{proof} First of all, according to Lemma 3.2 in \cite{Hebey}, one indeed has a lower bound for the Euclidean volume growth from the lower bound of the Yamabe constant.
Then from our Lemma \ref{improved} and the arguments in \cite{And1,BKN} it is known that the tangent cone $(M^n_\infty, \ g_\infty)$  at infinity for the Ricci flat manifold
$(M^n, \ g)$ is a flat cone $\mathbb{R}^n/\Gamma$. Because the blow-down: the rescaled manifolds $(M^n, \ \lambda_i^2 g, p)$,  as $\lambda_i\to 0$ with a fixed point $p\in M^n$,
converges to the tangent cone $(M_\infty, \ g_\infty, p_\infty)$ at infinity in Cheeger-Gromov topology away from the singular point $p_\infty$ (cf. \cite{And1,BKN}).
It is rather easily seen from \eqref{yamabe} that the Yamabe quotient of the blow-down sequence converges to that of the tangent cone
at infinity for any smooth function with compact support away from the vertex of the cone. For a smooth function with compact support in general, we simply use cut-off functions to modify
it as follows. For any small $s> 0$, we consider cut-off functions:
$$
\left\{\aligned \phi  & = 0 \text{ when } r\leq s\\
\phi & = 1 \text{ when } r\geq 2s\\
\phi & \in C^\infty\endaligned\right.
$$
We may require that $|\nabla \phi|\leq Cs^{-1}$ for some constant $C$. Then, for a function $u$ with compact support, which is smooth away from the vertex and Lipschitz across the vertex,
we calculate
$$
\aligned
\int_{\mathbb{R}^n/\Gamma} |\nabla u|^2 & = \int_{\mathbb{R}^n/\Gamma} |\nabla (\phi u)|^2 + \int_{\mathbb{R}^n/\Gamma}
(2\nabla (\phi u)\cdot \nabla((1 - \phi) u) + |\nabla ((1 - \phi) u)|^2)\\
& = \int_{\mathbb{R}^n/\Gamma} |\nabla (\phi u)|^2 + O(s^{n-2})\endaligned
$$
and
$$
\int_{\mathbb{R}^n/\Gamma} u^\frac {2n}{n-2} = \int_{\mathbb{R}^n/\Gamma} (\phi u)^\frac{2n}{n-2} + O(s^n)
$$
as $s\to 0$.
Therefore, due to the scaling invariance of the Yamabe constant, we also know that
$$
Y(M^n_\infty, \ g_\infty) \geq \liminf Y(M, \lambda_i g) \geq \tau Y(\mathbb{R}^n, g_E)
$$
for some $\tau > 2^{-\frac 2n}$ from our assumption \eqref{inequlocSob}. Suppose that
$$
D_\Gamma: \mathbb{R}^n\to \mathbb{R}^n/\Gamma
$$
is the desingularization of the cone metric. Then it is easily seen that
\begin{equation}\label{quotient}
\frac {\int_{\mathbb{R}^n} \frac {4(n-1}{n-2}(|\nabla u\circ D_\Gamma|^2dv)[g_E]}{(\int_{\mathbb{R}^n}(u\circ D_\Gamma)^\frac {2n}{n-2}dv[g_E])^\frac {n-2}n}  =  2^{\frac 2n}
\frac {\int_{M^n_\infty} \frac {4(n-1}{n-2}(|\nabla u|^2dv)[g_\infty]}{(\int_{M^n_\infty} u^\frac {2n}{n-2}dv[g_\infty])^\frac {n-2}n}.
\end{equation}
Hence, when using appropriately modifications of the standard function
$$
u \circ D_\Gamma= (\frac 2{1 + |x|^2})^\frac {n-2}2
$$
in \eqref{quotient}, one easily gets
$$
Y(M^n_\infty, \ g_\infty) \leq |\Gamma|^{-\frac 2n}Y(\mathbb{R}^n, \ g_E),
$$
which,  comparing to \eqref{inequlocSob}, forces the group $\Gamma$ to be trivial and the Ricci flat manifold $(M^n, \ g)$ to be isometric to the
Euclidean space $(\mathbb{R}^n, \ g_E)$ according Lemma \ref{improved} (\cite{CN}) in dimension 4 and Theorem \ref{thmasympflat} (\cite{And1,BKN})
in general dimensions. So the proof is complete.
 \end{proof}

Now we are ready to give the proof of our Theorem \ref{conformal-gap}:

\vskip 0.1in\noindent{\it Proof of Theorem \ref{conformal-gap}}: \quad
If not, then there exists a sequence of Bach flat 4-manifolds $(M^4_j, \ (g_Y)_j)$, that are not conformally equivalent to the round 4-sphere,  where $(g_Y)_j$ is the Yamabe metric
with $R[(g_Y)_j]=12$ and
$$
\int_M (\sigma_2(A) dv)[(g_Y)_j] = (1 - \epsilon_j)^2 16\pi^2 \to 16\pi^2$$
as $j\to \infty$.  Hence, in the light of \eqref{sigma-2}, we have
\begin{equation}\label{Sobolev-control}
(1 - \epsilon_j)Y(\mathbb{S}^n, \ [g_{\mathbb{S}}]) \leq Y(M^4_j, [(g_Y)_j]) \leq Y(\mathbb{S}^n, \ [g_{\mathbb{S}}])
\end{equation}
and
\begin{equation}\label{ricci-control}
\int_{M_j}(|\overset{\circ}{\text{Ric}}|^2dv)[(g_Y)_j] \leq \frac 1{12}\epsilon_j\to 0.
\end{equation}
Here we use the fact that, on the round sphere $(\mathbb{S}^4, g_{\mathbb{S}})$,
\begin{align*}
\int_{S^4} (\sigma_2(A) dv)[g_{\mathbb{S}}] = \frac{1}{24}Y(\mathbb{S}^4, \ [g_{\mathbb{S}}])^2.
\end{align*}
From the argument in \cite{CQY1}, to get a contradiction it suffices to show that $M^4_j$ is diffeomorphic to $\mathbb{S}^4$ for a subsequence of $j$.
More precisely, here we use the following interesting result from \cite{CQY1}:

\begin{lem}\label{from-cqy} (\cite{CQY1}) There is $\epsilon > 0$ such that a Bach flat metric $g$ on 4-sphere $\mathbb{S}^4$ is conformal to the standard round metric
$g_{\mathbb{S}}$, provided that
$$
\int_{\mathbb{S}^4} (|W|^2dv)[g] \leq \epsilon.
$$
\end{lem}

The rest of the proof of Theorem \ref{conformal-gap} follows from a more or less standard rescaling argument based on our Lemma \ref{trivialgroup}. We first derive a contradiction if
there was curvature blow-up. Again, we use Lemma 3.2 in \cite{Hebey} to get the uniform lower bound on the Euclidean volume growth for such sequence of manifolds.
Then, one stands at the point of curvature blow-up, that is, $p_j\in M_j$ such that
$$
\lambda_j = |Rm|(p_j)[(g_Y)_j] = \max_{M_j} |Rm|[(g_Y)_j] \to \infty
$$
and considers the sequence of pointed Riemannian manifold $(M_j, g_j, p_j)$ with the rescaled metric $g_j = \lambda^2_j (g_Y)_j$. Therefore, according to the curvature
estimates established, for example, in \cite{Tian-Viaclovsky}, one derives a subsequence that converges to complete non-compact  manifold $(M_\infty, g_\infty, p_\infty)$ in
Cheeger-Gromov topology. As the consequences of \eqref{Sobolev-control} and \eqref{ricci-control} one knows that
\begin{itemize}
\item $Y(M_\infty, g_\infty) = Y(\mathbb{S}^n, [g_{\mathbb{S}}])  =  Y(\mathbb{R}^n, g_E)$ and
\item $\overset{\circ}{\text{Ric}}[g_\infty] = 0$ and $R[g_\infty] = 0$.
\end{itemize}
Here we use the argument similar to that in the proof in Lemma \ref{trivialgroup} to derive the equality $Y(M_\infty, g_\infty) = Y(\mathbb{R}^n, g_E)$.
Therefore $(M_\infty, g_\infty)$ is isometric to the Euclidean 4-space according to Lemma \ref{trivialgroup} in the same time $|Rm|(p_\infty)[g_\infty] = 1$, which is a contradiction.
On the other hand, if there is no curvature blow-up for the sequence $(M_j, \ (g_Y)_j)$, then, the same compactness argument implies that there would be a subsequence
that converges to the round sphere in Cheeger-Gromov topology, which is impossible due to Lemma \ref{from-cqy}. Thus the proof of Theorem \ref{conformal-gap} is complete.
\qed \\

Next we use the facts collected in the section of preliminaries to give a proof of Theorem \ref{CGT}:

\vskip 0.1in\noindent{\it Proof of Theorem \ref{CGT}}: \quad
Since $g^{(3)}=0$, from the expansion \eqref{metric-expansion},  the doubling
$$(X_D, \ \tilde g) = (\bar X\bigcup \bar X, \bar g)$$ is a smooth Bach flat 4-manifold
(for more details about the doubling please see \cite{CQY2}, which uses \cite{Chang-Gursky-Yang}).  In fact, it is also shown in \cite{CQY2} that the doubling
$(X_D, \tilde g)$ is of positive Yamabe type from the assumption that the conformal infinity $(\partial X^3, [\hat g])$ is of positive Yamabe type.
In the mean time we recall from \eqref{sigma-2-gauss} that
$$
\int_{X_D} (\sigma_2(A)dv)[\tilde g] = 2 \int_X (\sigma_2(A)dv)[\bar g] = 12V(X^4, \ g^+)
$$
to conclude that
$$
\int_{X_D} (\sigma_2(A)dv)[\tilde g] \geq (1 - \epsilon)16\pi^2
$$
by the assumption \eqref{volume-gap}. Now one applies Theorem \ref{conformal-gap} to the doubling $(X_D, \ \tilde g)$ and derives that $(X_D, \tilde g)$ is conformally equivalent to
the round 4-sphere when $\epsilon$ is sufficiently small. Particularly one obtains that $(X^4, \ g^+)$ is a simply connected Riemannian manifold of constant curvature $-1$
(cf. \cite{CQY2}) when $\epsilon$ is sufficiently small, which completes the proof.
\qed
\\

To derive Theorem \ref{curvature-bound} we carry the above rescaling scheme with Lemma \ref{trivialgroup} on conformally compact Einstein manifolds.
In fact we are able to derive the curvature bound for conformally compact
Einstein manifolds in general dimensions, to which Theorem \ref{curvature-bound} is a corollary.

\begin{thm} For a constant $B$ and a constant $\tau > 2^{-\frac 2n}$ , there exists a constant  $C=C(n, \tau, B) >0$ such that
\begin{equation}\label{bound}
|Rm| \leq C,
\end{equation}
for any conformally compact Einstein manifold $(X^{n}, \ g^+)$ ($n\geq 4$) with
\begin{itemize}
\item $Y(X^{n}, \ g^+) \geq \tau Y(\mathbb{H}^{n}, \ g_{\mathbb{H}})$ and
\item $\int_{X^{n}} (|W|^\frac {n}2dv)[g^+] \leq B$ when $n\geq 5$.
\end{itemize}
\end{thm}

\begin{proof} Suppose otherwise that there is a sequence of conformally compact Einstein manifolds $(X^n_j, \ g^+_j)$ satisfying the assumptions in the theorem with curvature
blowing up. Since conformally compact Einstein manifolds are always asymptotically hyperbolic, we may extract a sequence of points $p_j\in X^n_j$ such that
$$
\lambda_j = |Rm|(p_j)[g^+_j] = \max_{X^n_j} |Rm|[g^+_j] \to \infty,
$$
and consider the pointed rescaled manifolds $(X^n_j, \ \lambda_j g^+_j, p_j)$. It is then easily seen that there is a subsequence $(X^n_j, g^+_j, p_j)$
converges in Cheeger-Gromov topology to a complete non-compact Ricci flat manifold $(X^n_\infty, g_\infty, p_\infty)$ with
\begin{itemize}
\item $Y(X^{n}_\infty, \ g_\infty) \geq \tau Y(\mathbb{H}^n, g_{\mathbb{H}}) = \tau Y(\mathbb{R}^{n}, \ g_E)$;
\item $|Rm|(p_\infty)[g_\infty] = 1$;
\item $\int_{X^{n}_\infty} (|W|^\frac {n}2dv)[g_\infty] \leq B$ when $n\geq 5$,
\end{itemize}
which is a contradiction in the light of Lemma \ref{trivialgroup}.
\end{proof}

\begin{cor} For $\epsilon < \frac 12$, there exists a number $C>0$ such that
$$
|Rm|[g^+] \leq C,
$$
for any conformally compact 4-manifold $(X^4, \ g^+)$ with conformal infinity of positive Yamabe type and
$$
V(X^4, \ g^+) \geq (1 - \epsilon)\frac {4\pi^2}3 = (1-\epsilon)V(\mathbb{H}^4, \ g_{\mathbb{H}}).
$$
\end{cor}
\begin{proof} It suffices to verify that
\begin{equation}\label{sobolev-need}
Y(X^4, g^+) \geq (1-\epsilon)^{\frac 12}Y(\mathbb{H}^4, g_{\mathbb{H}}).
\end{equation}
From \eqref{Sobolev-control} in the proof of Lemma \ref{trivialgroup} we know
$$
Y(X_D, \ \tilde g) \geq (1-\epsilon)^\frac 12 Y(\mathbb{H}^4, \ g_{\mathbb{H}}),
$$
which implies \eqref{sobolev-need} by the conformal invariance of the Yamabe constant.
\end{proof}

Clearly, the local Yamabe constant $Y(X^4, \ g^+)$ for a conformally compact Einstein 4-manifold approaches that of the hyperbolic space as the renormalized
volume $V(X^4, \ g^+)$ approaches that of the hyperbolic space. Before we end this section we state an easy observation based on the construction of Aubin
on the impact to local Yamabe constant from local geometry in higher dimensions (cf. Paragraph 6.10 in \cite{Aubin, LP}).

\begin{prop} \label{higher-dim} For any $\epsilon >0$, there is $\delta >0$ such that, for any section curvature $K$ at any point
on any conformally compact Einstein manifold $(X^n, \ g^+)$ ($n\geq 6$), one has
$$
|K + 1| \leq \epsilon,
$$
provided that
$$
Y(X^n, \ g^+) \geq (1- \delta)Y(\mathbb{H}^n, \ g_{\mathbb{H}}).
$$
\end{prop}

\begin{proof} Assume otherwise, for some $\epsilon_0 > 0$,  there is a sequence of conformally compact Einstein manifolds $(X^n_j, \ g^+_j)$ and a sequence
$\delta_j \to 0$ such that
$$
Y(X^n_j, \ g^+_j) \geq (1 - \delta_j)Y(\mathbb{H}^n, \ g_{\mathbb{H}}),
$$
but
$$
|K+1| > \epsilon_0,
$$
for a sectional curvature $K$ at some point on $X^n_j$. We are going to derive contradictions in two steps. First, if there is a subsequence $(X^n_j, \ g^+_j)$ whose
sectional curvature is not bounded as $j\to \infty$ (for connivence, we continue to use the same index $j$), then we may rescale the metrics
$$
\tilde g_j = \lambda_j^2 g^+_j
$$
for $\lambda_j =|W|(p_j)[g^+_j] =  \max_{X^n_j} |W|[g^+_j]$. Then, due to the curvature estimates for Einstein manifolds, it is easily seen that, at least for a subsequence,
$(X^n_j, \ \tilde g_j, p_j)$ converges to a complete non-compact Ricci flat manifold $(X^n_\infty, \ g_\infty, p_\infty)$ in Cheeger-Gromov topology. But
$|W|(p_\infty)[g_\infty] = 1$ contradicts with the fact that $Y(X^n_\infty, \ g_\infty) = Y(\mathbb{R}^n, g_E)$ in the light of the estimate of Yamabe constant in Paragraph 6.10
in \cite{Aubin} (cf. also \cite{LP}). \\

Secondly, if there is no curvature blowup, then one may extract a subsequence such that
$$
|W|(p_j)[g^+_j] \to w > 0
$$
and $(X^n_j, \ g^+_j, p_j)$ converges to a complete non-compact Einstein manifold $(X^n_\infty, \ g^+_\infty, p_\infty)$ in Cheeger-Gromov topology. But, again,
$|W|(p_\infty)[g^+_\infty] = w > 0$ contradicts with the fact that $Y(X^n_\infty, \ g^+_\infty) = Y(\mathbb{H}^n, g_{\mathbb{H}})$ in the light of the estimate of
Yamabe constant in Paragraph 6.10 in \cite{Aubin}(cf. also \cite{LP}).
\end{proof}
One wonders whether Proposition \ref{higher-dim}  still holds in dimension 4, which would be much more significant.


\section{Conformally compact Einstein manifolds whose conformal infinities have large Yamabe constants}\label{Sect:large-yamabe}

In this section we will first present the idea in \cite{DJ} to establish the relative volume growth bounds \eqref{volume-com}. We recognize the contribution from \cite{DJ} but are compelled
to give self-contained arguments for a complete, vigorous and correct proof of \eqref{volume-com} to the best of our knowledge.
We will point out what are not clear and not correct in \cite{DJ} and finish filling those
gaps in the subsequent sections. Then we will carry out the rescaling argument to derive the curvature estimates in Theorem \ref{EHBoundary} in two steps similar to that in the proof of
Proposition \ref{higher-dim}.\\

We recall that a Riemannian manifold $(X^n, \ g^+)$ is said to be AH (short for asymptotically hyperbolic) if it is conformally compact and its curvature goes to $-1$ at the infinity.
Obviously a conformally compact Einstein manifold is always AH.
Let us consider the distance function to a given point $p_0\in X^n$:
$$
t = \text{dist}(\cdot, p_0)
$$
and the geodesic sphere $\Gamma_t=\{p\in X^n: \text{dist}(p, p_0) = t\}$.
The important initial step is the estimate (2.3) in Lemma 2.1 of \cite{DJ}, which is
\begin{equation}\label{2.3-dutta}
\nabla_g^2t (v, v) = 1 + O(e^{-\beta t}),
\end{equation}
where $v$ is any unit vector perpendicular to $\nabla t$ and $\beta$ is any positive number less than $2$. It is not clear to us how Section 6.2 in \cite{petersen} is applied in the proof
\eqref{2.3-dutta} in \cite{DJ}. In fact the estimate \eqref{2.3-dutta} does not seem to be correct without the convexity of geodesic spheres (cf. \cite{ST,HQS}). \\

It is observed in \cite{DJ} that one may employ the Bishop-Gromov relative volume comparison theorem and get
$$
\frac {\text{Vol}(\Gamma_t, g^+)}{\text{Vol}(\Gamma_t, g_{\mathbb{H}})} \leq \frac {\text{Vol}(B(t, p_0), g^+)}{\text{Vol}(B(t, 0), g_{\mathbb{H}})} \leq 1
$$
for all $t > 0$ on a conformally compact Einstein manifold. Hence the real issue for \eqref{volume-com} is the lower bound and the key is to establish the relative volume
lower bound by the limit
$$
\lim_{t\to\infty}  \frac {\text{Vol}(\Gamma_t, g^+)}{\text{Vol}(\Gamma_t, g_{\mathbb{H}})}.
$$
It is very original in \cite{DJ} to realize that one may use the Yamabe quotient to bound the relative volume $ \frac {\text{Vol}(\Gamma_t, g^+)}{\text{Vol}(\Gamma_t, g_{\mathbb{H}})}$
from below. Suppose that $(X^n, \ g^+)$ is an AH manifold and that $x$ is the geodesic defining function associated with a representative $\hat g$ of the conformal infinity
$(\partial X^{n-1}, \ [\hat g])$. Let
\begin{equation}\label{r-def}
r = -\log \frac x2
\end{equation}
and $\Sigma_r = \{p\in X^n: r(p) = r\}$ be the level set of the geodesic defining function. We would like to mention that the inequality
(2.2) in Lemma 2.1 of \cite{DJ} is a well known fact about AH manifolds of $C^3$ regularity, which is
\begin{equation}\label{2.2-dutta}
Ddr(v,v) = 1 + O(e^{-2r}),
\end{equation}
where $v$ is any unit vector perpendicular to $\nabla r$, provided that \eqref{cur-decay} holds.
Let $\bar g = x^2g^+$ be the conformal compactification from the defining function and let
$$\bar g_t = \bar g|_{\Gamma_t} \text{ and }\bar g_r=\bar g|_{\Sigma_r}.$$
Also let $\tilde g = 4e^{-2t}g^+ = \psi^\frac 4{n-3}\bar g$ be the conformal compactification from the distance function and let
$$\tilde g_t =  \tilde g|_{\Gamma_t}  \text{ and } \tilde g_r= \tilde g|_{\Sigma_r},
$$
where $\psi=e^{\frac{n-3}{2}(r-t)}$. First of all it is easily seen that $u = r-t$ is bounded on $X^n$ by the triangle inequality for distance functions. It is then observed in Lemma 3.1 of
\cite{DJ} that $|\nabla u|[\bar g]$ is bounded. For the convenience of readers we present a complete proof and an easy fix for a gap in \cite{DJ}.

\begin{lem}\label{liptschitz} Suppose that $(X^n, \ g^+)$ is an AH manifold of $C^3$ regularity and \eqref{cur-decay} holds. Then there is a constant $C$ such that
\begin{equation}\label{gradient-u}
|d u|[\bar g] \leq C,
\end{equation}
when $r$ is large enough.
\end{lem}

\begin{proof} It suffices to show that
\begin{equation}\label{angle-r-t}
g^+(\nabla t, \nabla r) = 1 + O(e^{-2t}).
\end{equation}
First, as noticed in the proof of Lemma 3.1 in \cite{DJ},  if let $\phi = g^+(\nabla t, \nabla r)$, then
\begin{equation}\label{phi'}
\frac d{dt} \phi = \nabla^2 r(\nabla t, \nabla t) =  (1 - \phi^2) \nabla^2 r(v, v),
\end{equation}
where $\nabla t = \sqrt{1-\phi^2}v + \phi \nabla r$.  We want to point out that it is not enough to derive \eqref{angle-r-t} just from
\eqref{phi'} and \eqref{2.2-dutta}. One needs the next lemma, which turns out to an easy fact but not a consequence of \eqref{angle-r-t}
as presented in the proof of Lemma 4.1 of \cite{DJ}.
\end{proof}

\begin{lem}\label{right-sign} Suppose that $(X^n, \ g^+)$ is AH of $C^3$ regularity and \eqref{cur-decay} holds. Let $x$ be the geodesic defining function associated
with a representative $\hat g$ of the conformal infinity $(\partial X, \ [\hat g])$.
Let $p_0\in X^n$ is a fixed point and $t$ be the distance to $p_0$ in $g^+$. And let $r$ be given in \eqref{r-def}. Then
$$
g^+(\nabla r, \nabla t) > 0
$$
at the point $t$ is smooth and $r$ is sufficiently large.
\end{lem}

\begin{proof} It is known (cf. (2.2) in Lemma 2.1 of\cite{DJ}) that, there is $r_0>0$ such that, for $r>r_0$,
$$
\nabla^2 r(v, v) > \frac 12
$$
for all unit vector $v$ that is perpendicular to $\nabla r$.  Then we claim that
$$
\phi = g^+(\nabla t, \nabla r) > 0
$$
for any point $p$ where $t$ is smooth and $r> r_0$. To see this, one considers the minimal geodesic $\gamma$ that connects $p$ to $p_0$ and realizes the distance. Then,
in the light of $(\ref{phi'})$, it is not hard to see that $\phi > 0$ from the time $t_0$ when the geodesic $\gamma$ exits from $\Sigma_{r_0}$. Because $\phi\leq 1$
for all $t$ and $\phi(t_0)\geq 0$.
\end{proof}

By Lemma \ref{right-sign}, for $r$ large, a geodesic
from $p_0$ can touch $\Sigma_r$ at most once till the time it hits the cut locus of $p_0$. Hence by $(\ref{gradient-u})$ the metric $\tilde g = \psi^\frac 4{n-3}\bar g$ extends to $\overline{X^n}$ with Lipschitz regularity up to the boundary.
In fact, as a consequence of the estimate \eqref{2.2-dutta} and Lemma \ref{liptschitz} as observed in Corollary 2.2 of \cite{DJ}, one has the following:

\begin{cor} \label{cor-2.2} Suppose that $(X^n, \ g^+)$ is AH of $C^3$ regularity and \eqref{cur-decay} holds.
Then the second fundamental form of the level set $\Sigma_r$ in $(X^n, \ \tilde g)$ converges to zero as $r\to\infty$.
\end{cor}

Corollary \ref{cor-2.2} is used in the calculation of the scalar curvature $R[\tilde g_r]$ for $(\Sigma_r, \tilde g_r)$ (cf. (5.5) in \cite{DJ}).
The following is one of the key steps in \cite{DJ}. We will present the idea from \cite{DJ} and point out what are not clear and not correct . We will finish filling those gaps in the
subsequent sections.

\begin{lem}\label{lemboundarySobolev}
Suppose that $(X^n, \ g^+)$ is an AH manifold of $C^3$ regularity and $(\ref{alter-einstein})$ holds, with its conformal infinity $(\partial X^{n-1}, [ \hat{g}])$ of positive Yamabe type. Then
\begin{align*}
(\frac{Y(\partial X, [\hat{g}])}{(n-2)(n-1)})^{\frac{n-1}{2}}\leq \text{Vol}(\partial X, \tilde g_0) = \lim_{t\to \infty} \text{Vol}(\Gamma_t, \tilde g_t),
\end{align*}
where $\tilde g_0$ is the continuous extension of $\tilde g$ to the boundary.
\end{lem}

\begin{proof}  We are recapturing the proof given in \cite{DJ} in the way that the use of the estimate \eqref{2.3-dutta} and the places where
more vigorous arguments are required are explicitly identified.  The full proof of this lemma will be completed in the subsequent sections.
\\

The first step is to derive (3.8) in \cite{DJ}. Using the Lipschitz extension of $\psi$ based on Lemma \ref{liptschitz}, we have
\begin{equation}\label{r-level-set}
\frac {\int_{\partial X} ((\frac {4(n-2)}{n-3}|\nabla \psi|^2 + R\psi^2)dv)[\hat g]}{(\int_{\partial X} \psi^\frac {2(n-1)}{n-3}dv[\hat g])^\frac {n-3}{n-1}} \leq \liminf_{r\to\infty}
\frac {\int_{\Sigma_r} ((\frac {4(n-2)}{n-3}|\nabla \psi|^2 + R\psi^2)dv)[\bar g_r]}{(\int_{\Sigma_r} \psi^\frac {2(n-1)}{n-3}dv[\bar g_r])^\frac {n-3}{n-1}}.
\end{equation}
On the other hand, we remark that, though
\begin{equation}\label{gap-1}
\int_{\Sigma_r} ((\frac {4(n-2)}{n-3}|\nabla \psi|^2 + R\psi^2)dv)[\bar g_r] = \int_{\Sigma_r} (Rdv)[\tilde g_r]
\end{equation}
easily holds when $\psi$ is smooth, \eqref{gap-1} (cf. (3.7) in \cite{DJ}) may not be correct when $\psi$ is only known to be Liptschiz. Our approach is to use
the deep understanding of the structure of cut loci to overcome the challenge based on \cite{Ozols, Tanaka}. We will deal with this issue and complete this first step
in Section \ref{Sect:normal-cut-point} (cf. Theorem \ref{step-one}).
\\

The second step is to obtain the pointwise scalar curvature estimate
\begin{equation}\label{scalar-cur}
R[\tilde g_r] \leq (n-1)(n-2) + o(1)
\end{equation}
(cf. (5.1) in \cite{DJ}).
The proof of Lemma 5.1 in \cite{DJ} uses Corollary \ref{cor-2.2} and the Laplacian comparison theorem.  More importantly
it uses the estimate \eqref{2.3-dutta} in calculating
$$
\nabla^2t(\nabla r, \nabla r) = 1 - (g^+(\nabla r, \nabla r))^2 +  O(e^{-3t}),
$$
which indeed would imply \eqref{scalar-cur} and
\begin{equation}\label{total-cur-need}
\liminf_{r\to\infty} \frac
{\int_{\Sigma_r} (Rdv)[\tilde g_r]}{(\int_{\Sigma_r}dv[\tilde g_r])^\frac {n-3}{n-1}}
\leq (n-1)(n-2) \lim_{r\to\infty} \text{Vol}(\Sigma_r, \tilde g)^\frac 2{n-1} = (n-1)(n-2) \text{Vol}(\partial X, \tilde g_0)^\frac 2{n-1}.
\end{equation}
We will present a proof of \eqref{total-cur-need} without assuming the estimate \eqref{2.3-dutta} at each smooth point of $t$ on $\Sigma_r$.
Our key idea is to show the part of $\Sigma_r$ where \eqref{2.3-dutta} does not holds has arbitrarily small volume. We will present a complete and correct proof
the second step in Section \ref{Sect:curvature-estimate} (cf. Theorem \ref{step-two}).
\\

The last step is to show that
\begin{equation}\label{volume-geodesic-sphere}
\lim_{t\to\infty} \text{Vol}(\Gamma_t, \tilde g_t) = \text{Vol}(\partial X, \tilde g_0).
\end{equation}
To do so, similar to the discussions in Section 6 of \cite{DJ} that is based on Lemma \ref{liptschitz}, one considers the geodesic sphere $\Gamma_t$ as a Lipschitz graph
over $\partial X$ in $(X^n, \ \tilde g)$, where
$$
x(t) \to 0 \text{ as $t\to\infty$}
$$
in $W^{1, p}$ topology for any $p\in [1, \infty)$. It is then easily seen that \eqref{volume-geodesic-sphere} holds, as shown in \cite{DJ}.
\end{proof}

Consequently, from the idea in \cite{DJ},  for conformally compact Einstein manifolds, we have the lower bound of the relative volume growth. \\

\noindent{\it Proof of Theorem \ref{main-dutta}} One only needs to realize that, as calculated in \cite{DJ}, the following:
\begin{align}\label{ratio}
\frac{\text{Vol}(\Gamma_t, g^+)}{\text{Vol}(\Gamma_t, g_{\mathbb{H}})} = \frac{\text{Vol}(\Gamma_t, \tilde g_t)(\frac{e^{t}}{2})^{n-1}}{\omega_{n-1}\sinh^{n-1}t}
= \frac{\text{Vol}(\Gamma_t, \tilde g_t)}{\omega_{n-1}}+o(1)
\end{align}
as $t\to \infty$.
\qed \\

Now we are ready to prove Theorem \ref{EHBoundary}.\\

\noindent{\it Proof of Theorem \ref{EHBoundary}}. First we want to show that, there are constant $\delta_0> 0$ and $C$ such that
$$
|W|[g^+]\leq C
$$
for any conformally compact Einstein manifolds that satisfy the assumptions  in Theorem \ref{EHBoundary} for $0< \delta \leq \delta_0$.
Assume otherwise there is a sequence of conformally compact Einstein manifolds $(X^n_j, \ g^+_j)$ such that
$$
|W|[g^+_j]\to \infty \text{ and } Y(\partial X_j, [\hat g_j]) \to Y(\mathbb{S}, [g_{\mathbb{S}}])
$$
as $j\to\infty$. By Theorem \ref{main-dutta}, we know that
\begin{equation}\label{j-volume}
(\frac{Y(\partial X_j, [\hat{g}_j])}{Y(\mathbb{S}, [g_{\mathbb{S}}])})^{\frac{n-1}{2}}\leq\frac {\text{Vol}(\Gamma_t, g^+_j)}{\text{Vol}(\Gamma_t, g_{\mathbb{H}})}
\leq \frac {\text{Vol}(B(p_j, t), g^+_j)}{\text{Vol}(B(0, t), g_{\mathbb{H}})}  \leq 1
\end{equation}
for $t>0$. Since $|W|[g^+_j](p)\to 0$ as $p\to \infty$ on each conformally compact Einstein manifold $(X^n_j, \ g^+_j)$, there exists a point $p_j \in X^n_j$ so that
$$
\tau_j =  |W|[g^+_j](p_j) = \max_{p\in X^n_j}|W|[g^+_j](p)\to\infty
$$
as $j\to\infty$. We then consider the rescaled metric $g_j = \tau_j g^+_j$ on the pointed manifold $(X^n_j, p_j)$. From \eqref{j-volume}, one may conclude
that the sequence of pointed Einstein manifolds $(X^n_j, \tau_jg^+_j, p_j)$ converges to a Ricci flat manifold $(X^n_\infty, g_\infty, p_\infty)$ in Cheeger-Gromov
topology. In particular, one gets, again from \eqref{j-volume},
$$
\text{Vol}(\Gamma_t, g_\infty) = \text{Vol}(\Gamma_t, g_{E})
$$
for all $t > 0$, which implies that $(X^n_\infty, g_\infty)$ is isometric to the Euclidean space $(\mathbb{R}^n, \ g_{E})$ and hence contradicts with
$|W|(p_\infty)[g_\infty] = 1$. \\

To finish the proof of Theorem \ref{EHBoundary} we assume again otherwise, there are $\epsilon_0> 0$ and a sequence of conformally compact Einstein
manifolds $(X^n_j, \ g^+_j)$ such that
$$
|W|(p_j)[g_j^+] \geq \epsilon_0 \text{ and } Y(\partial X, \ [\hat g_j]) \to Y(\mathbb{S}^n, [g_{\mathbb{S}}])
$$
as $j\to\infty$. According to the above uniform bound for the curvature for such a sequence, we may extract a subsequence of pointed Einstein manifolds $(X^n_j, \ g^+_j, p_j)$
with $|W|(p_j)[g^+_j]\geq \epsilon_0 > 0$, which converges to an Einstein manifold $(X^n_\infty, g_\infty, p_\infty)$ in Cheeger-Gromov topology. Then the exact same argument
at the end of Section 7 in \cite{DJ} produces a contradiction from $|W|(p_\infty)[g_\infty] \geq \epsilon_0 > 0$. So the proof is complete.
\qed
\\

It is obvious that the argument in the first step in the above proof of Theorem \ref{EHBoundary}
implies Corollary \ref{CCEflat}, in the same spirit as in the proof of Lemma \ref{trivialgroup}.


\section{Normal cut loci}\label{Sect:normal-cut-point}

In this section we focus on the issue in the first step of the proof of Lemma \ref{lemboundarySobolev}. First $\psi$ is smooth away from the cut loci of the point $p$
in a conformally compact Einstein manifold $(X^n, \ g^+)$. Hence it is necessary to understand the fine structure of the set of cut loci and the behavior of the distance
function near the cut loci in order to study \eqref{gap-1}. \\

One might think, for a fixed $p$ in a complete and non-compact manifold $(M^n, \ g)$, the cut loci $C_{p}$ may stay in a compact set. In fact, to the contrary, any component of
$C_p$ extends to the infinity, unless $p$ is a pole. This is because $M^n\setminus C_p$ is always diffeomorphic to the Euclidean space $\mathbb{R}^n$.
\\

On a complete Riemannian manifold $(M^n, \ g)$ with a fixed point $p\in M^n$, the set $C_p$ of cut loci consists of the set $Q_p$ of conjugate points and the set
$A_p$ of non-conjugate cut loci. Among the points in $A_p$ we call those from which there are exactly two minimal geodesics connecting to $p$ and realizing
the distance to $p$ in $(M^n, \ g)$ the normal cut loci, according to \cite{Ozols, Tanaka}. We denote the set of all normal cut loci by $N_p$ and
the rest of non-conjugate cut loci by $L_p$. In those notations we have
$$
C_p = Q_p\bigcup L_p \bigcup N_p.
$$
We recall from \cite{Ozols, Tanaka} the following facts about the cut loci on Riemannian manifolds in general.

\begin{lem} \label{Ozols} (\cite{Ozols, Tanaka}) Suppose that $(M^n, \ g)$ is a complete Riemannian manifold and that $p\in M^n$. Then
\begin{itemize}
\item The closed set $Q_p\bigcup L_p$ is of Hausdorff dimension no more than $n-2$.
\item The set $N_p$ of normal cut loci consists of possibly countably many disjoint smooth hypersurfaces in $M^n$.
\item Moreover, at each normal cut locus $q\in N_p$, there is a small open neighborhood $U$ of $q$ such that
$U\bigcap C_p = U\bigcap N_p$ is a piece of smooth hypersurface in $M^n$.
\end{itemize}
\end{lem}

In our cases, on a conformally compact Einstein manifold $(X^n, \ g^+)$ with a given point $p$, we are concerned with the set
$$
\gamma_r = \Sigma_r \bigcap C_{p} = (\Sigma_r\cap (Q_{p}\cup L_{p}))\bigcup (\Sigma_r\cap N_{p}) = \gamma^{QL}_r\bigcup\gamma^N_r,
$$
where $\psi$ is not smooth as a Liptschitz function on $\Sigma_r$. Before we move to look closely on \eqref{gap-1} we would like to mention some more
facts about the distance function $t$ and the geodesic spheres in our context.

\begin{lem} \label{less-pi} Suppose that $(X^n, \ g^+)$ is AH of $C^3$ regularity and \eqref{cur-decay} holds. Then
\begin{itemize}
\item When $t$ is sufficiently large, the geodesic sphere $\Gamma_t$ is a Liptschitz graph over $\partial X$.
\item When $t$ is sufficiently large, the outward angle of the corner at the normal cut locus on $\Gamma_t$ is always less than $\pi$.
\end{itemize}
\end{lem}

\begin{proof} The first statement can be proven using the same argument as in Section 4 of \cite{DJ} (cf. Lemma \ref{liptschitz} and Lemma \ref{right-sign} in the previous
section).  \\

From \cite{Ozols, Tanaka} we know that the singularities for the geodesic sphere at normal cut loci are corners that are, at least locally,  the meet of two smooth hypersurfaces.
To see the outward angle of such corner at each normal cut locus is always less than $\pi$, let us assume otherwise. Let $\gamma$ be a (distance realizing) minimal geodesic
from the fixed $p$ to a normal cut locus $q\in\Gamma_t$ where the inward angle is less than $\pi$. We may push toward the geodesic sphere $\Gamma_t$
from inside a small geodesic ball centered along $\gamma$. Clearly the geodesic ball will definitely touch the geodesic sphere $\Gamma_t$ at some point $\bar q\in\Gamma_t$
before it reaches the corner $q\in \Gamma_t$. This yields a contradiction because the distance from $p$ to the point $\bar q\in\Gamma_t$ would be definitely
less than $t$.
\end{proof}

By compactness of $\gamma_r^{QL}$, there are finitely many points $p_i\in \gamma_r^{QL}$ so that $B_r^{\epsilon}=\bigcup_i B_{\epsilon}(p_i)$ covers $\gamma_r^{QL}$. Then $B_r^{\epsilon}$ together with $\gamma^N_r$ covers the set  $\gamma_r$. Then we perform the integral by parts
to calculate the Yamabe functional for $\psi$ as follows:
\begin{equation}\label{integral-by-part}
\aligned \int_{\Sigma_r\setminus (B^\epsilon_r\cup\gamma_r^N)}  & ((\frac {4(n-2)}{n-3}|\nabla\psi|^2 + R\psi^2)dv)[\bar g_r]
= \int_{\Sigma_r\setminus (B^\epsilon_r\cup\gamma_r^N)}  (R dv)[\tilde g_r] \\ & + \oint_{\partial (\Sigma_r\setminus (B^\epsilon_r\cup\gamma_r^N))}
\frac {4(n-2)}{n-3}\psi (\frac {\partial}{\partial n}\psi d\sigma)[\bar g_r]\endaligned
\end{equation}
In the light of Lemma \ref{Ozols}, for almost every $r$ as $r\to\infty$ and almost every $\epsilon$ as $\epsilon\to 0$, one may assume $\gamma^{QL}_r$ is of Hausdorff dimension
no more than $n-3$ and $\gamma^N_r\setminus B^\epsilon_r$ is a union of finitely many disjoint hypersurfaces in $\Sigma_r$. Hence
$$
\oint_{\partial (\Sigma_r\setminus (B^\epsilon_r\cup\gamma_r^N))}
\psi( \frac {\partial}{\partial n}\psi d\sigma)[\bar g_r]  = \oint_{\partial B^\epsilon_r} \psi (\frac {\partial}{\partial n}\psi d\sigma)[\bar g_r]
+ \oint_{\gamma_r^N\setminus B^\epsilon_r}
\psi ((\frac {\partial}{\partial n^+}\psi +\frac {\partial}{\partial n^-}\psi) d\sigma)[\bar g_r]
$$
where $n^+$ and $n^-$ are the two outward normal directions to $\gamma^N_r$ from the inside of $\Sigma_r\setminus (B^\epsilon_r\cup\gamma^N_r)$. It is not hard to see that
\begin{equation}\label{lower-hausdorff-dim}
\oint_{\partial B^\epsilon_r} \psi (\frac {\partial}{\partial n}\psi d\sigma)[\bar g_r] \to 0
\end{equation}
as $\epsilon\to 0$. Because, $\gamma^{QL}_r$ is compact and of Hausdorff dimension no more than $n-3$; and $\psi$ is uniformly Liptschitz on $\Sigma_r$ (cf. \eqref{gradient-u} in
Lemma \ref{liptschitz}).
Recall that
$\psi = e^{\frac {n-3}2 (r-t)}$ at least for almost every $r$. Therefore
$$
\frac {\partial}{\partial n^\pm}\psi =  - \frac {n-3}2 \psi \frac {\partial t}{\partial n^\pm} =  - \frac {n-3}2 \psi (\nabla t)^\pm \cdot n^\pm
$$
and
\begin{equation}\label{normal-derivatives}
\frac {\partial}{\partial n^+}\psi +\frac {\partial}{\partial n^-}\psi = -\frac {n-3}2\psi (\frac {\partial t}{\partial n^+} +\frac {\partial t}{\partial n^-} )
= -\frac {n-3}2\psi ((\nabla t)^+ - (\nabla t)^-)\cdot n^+,
\end{equation}
where $(\nabla t)^\pm$ is the gradient of the distance function $t$ with respect to the metric $\bar g$ from either side of the corner $\gamma^N_r$.

\begin{lem}\label{angle} For almost every $r$, when $\gamma^N_r$ is union of disjoint hypersurafces in $\Sigma_r$,
\begin{equation}\label{angle-less-half-pi}
\frac {\partial t}{\partial n^+} +\frac {\partial t}{\partial n^-} \geq 0
\end{equation}
at each point on $\gamma^N_r$.
\end{lem}

\begin{proof} Given a point $q\in \gamma^N_r$,  let us consider the plane P spanned by $(\nabla t)^+$ and $(\nabla t)^-$. And let $n^t$ be the unit direction of the
projection of $n^+$ to the plane P. One notices that, from \eqref{angle-r-t}, the angle between $(\nabla t)^+$ and $(\nabla t)^-$ is arbitrarily small as well as
the angle between $(\nabla t)^\pm$ and $n^t$ is arbitrarily close to $\frac \pi 2$, as $r\to \infty$. In the light of \eqref{normal-derivatives},
to verify \eqref{angle-less-half-pi} is to verify that the angle from $(\nabla t)^-$ to
$n^t$ is not smaller than the one from $(\nabla t)^+$ to $n^t$, since $\|(\nabla t)^+\| = \|(\nabla t)^-\|$.

\vskip 0.2in
\hskip 1.2in\includegraphics[scale= 0.5]{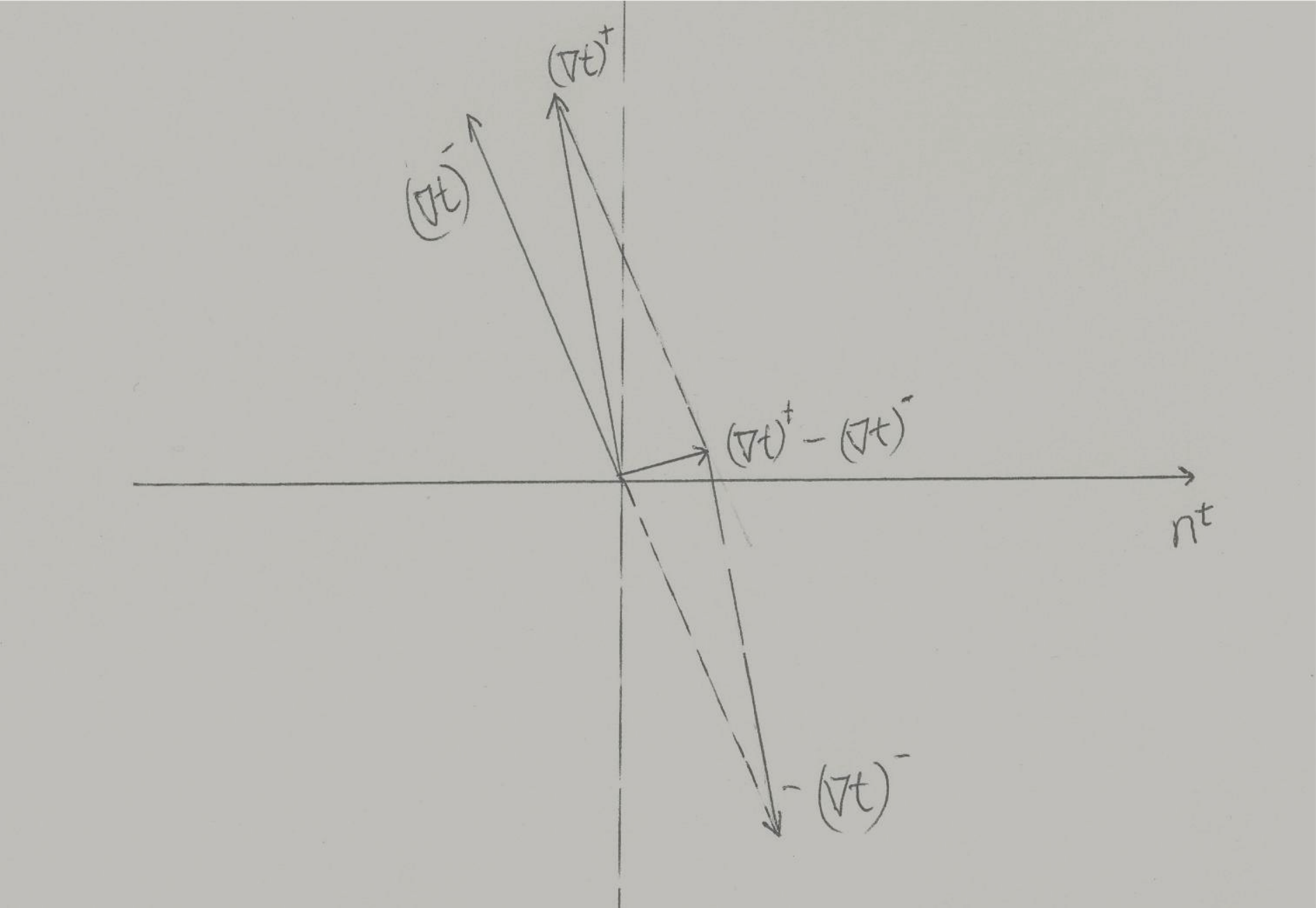}
\vskip 0.2in

This turns out to be true because the outward angle of the corner at any normal cut locus on the geodesic sphere is always less than $\pi$ according to Lemma \ref{less-pi}.
\end{proof}

From the proof of Lemma \ref{angle}, we can observe that for $r>0$ large, $N_p$ intersects with $\Sigma_r$ transversely, so that for $q\in N_p$, there exists $v\in T_q N_p$ so that the angle between $v$ and $\nabla_g r(q)$ is bounded by $Ce^{-r}$ with a uniform constant $C>0$. Therefore, $\mathcal {H}^{n-1}(\Sigma_r\bigcap C_p)=0$. Similarly, $\mathcal {H}^{n-1}(\Gamma_t\bigcap C_p)=0$ for $t>0$ large. To summarize what we have so far in this section we state the following proposition:

\begin{thm}\label{step-one} Suppose that $(X^n, \ g^+)$ is an AH manifold of $C^3$ regularity and \eqref{cur-decay} holds.
Let $p$ be a fixed point on $X^n$. Let $t$ be the distance function to $p$ on $(X^n, \ g^+)$
and let $r = - \log \frac x2$, where $x$ is the geodesic defining function associated with a representative $\hat g$ of the conformal infinity $(\partial X, \ [\hat g])$.
Then for almost all large $r>0$ so that $\mathcal {H}^{n-2}(\gamma_r^{QL})=0$, it holds that
\begin{equation}\label{with-singularity}
\int_{\Sigma_r} ((\frac {4(n-2)}{n-3}|\nabla\psi|^2 + R\psi^2)dv)[\bar g_r]
\leq  \int_{\Sigma_r\setminus (B^\epsilon_r\cup\gamma_r^N)}  (R dv)[\tilde g_r] + o_\epsilon(1),
\end{equation}
where $\psi = e^{\frac {n-3}2(r-t)}$ and $o_\epsilon(1)\to 0$ as $\epsilon\to 0$.
\end{thm}


\section{Estimates of the total scalar curvature}\label{Sect:curvature-estimate}

In this section we focus on the issue in the second step of the proof of Lemma \ref{lemboundarySobolev}. Let us first be very clear on how \eqref{2.3-dutta} is used in
the argument in \cite{DJ} and what one can hope to get for a upper bound for the scalar curvature $R[\tilde g_r]$.
Recall from \cite{DJ}, to estimate the scalar curvature $R[\tilde g_r]$, one starts with (5.5) in \cite{DJ},
$$
R[\tilde g_r] = R[\tilde g] - 2 \text{Ric}[\tilde g](N, N) + o(1)
$$
and
$$
\text{Ric}[\tilde g](N, N) = \frac {e^{2t}}4( \text{Ric}[g^+] (\nabla r, \nabla r) + \Delta t  + (n-2)(1 - (g^+(\nabla r, \nabla t))^2)(\nabla^2t(v, v)-1))
$$
where $N = \frac 12 e^t \nabla r$ and
$$\nabla r = g^+(\nabla r, \nabla t)\nabla t + \sqrt{1 - (g^+(\nabla r, \nabla t))^2}v$$ for some unit vector $v\perp\nabla t$. Hence, following the calculation in \cite{DJ} and
assuming
\begin{equation}\label{scalar-cur-asym}
\text{Ric}[g^+] \geq -(n-1)g^+ \text{ and } \ R[g^+] = - n(n-1) + o(e^{-2t}),
\end{equation}
one arrives at
\begin{equation}\label{estimate-0}
R[\tilde g_r] \leq (n-1)(n-2) + \frac {n-2}2 e^{2t}(1 - (g^+(\nabla r, \nabla t))^2)(1 - \nabla^2t(v, v)) + o(1),
\end{equation}
which implies \eqref{scalar-cur} whenever \eqref{2.3-dutta} is available. In fact it is clear that, any lower bound
of the principal curvature would yield a upper bound for the scalar curvature $R[\tilde g_t]$ from \eqref{estimate-0} and \eqref{angle-r-t}. To our best knowledge,
one does not have any a prior lower bound of the principal curavture, though one indeed can manage to get the desired upper bound for the principal curvature,
by the nature of the Riccati equations in general (cf. \cite{ST, HQS}). What we observe is that, smaller the principal curvature is; smaller the mean curvature is;
and therefore smaller the surface volume element is, at any smooth point of the distance function. Thus we will still be able to obtain the desired estimate
for the total scalar curvature
\begin{equation}\label{gap-2}
\int_{\Sigma_r\setminus C_{p_0}}(Rdv)[\tilde g_r] \leq (n-1)(n-2)\int_{\Sigma_r}dv[\tilde g_r] + o(1)
\end{equation}
where $o(1)\to$ as $r\to\infty$. We will organize this section into three subsections.

\subsection{Curvature estimates based on the Riccati equations}\label{riccati-equation} In this section we would like to derive the curvature estimates based on the
Riccati equations on AH manifolds. Let us start with the Riccati equation for the shape operator of the geodesic spheres along a geodesic $\gamma (t)$
(cf. for example, (2.1) in \cite{DJ} or \cite{ST, HQS, petersen})
\begin{equation}\label{riccati}
\nabla_{\nabla t} S + S^2 = - (R_{\nabla t})[g^+]
\end{equation}
where $S=\nabla^2t$ is the shape operator and $(R_{\nabla t})[g^+] (v) = (R(v, \nabla t)\nabla t)[g^+]$. Particularly, one has
\begin{equation}\label{riccati-1}
1 - C_0e^{-2t} \leq \mu'(t) + \mu^2(t) \leq 1 + C_0e^{-2t}
\end{equation}
for the principal curvature $\mu$, on an AH manifold with \eqref{cur-decay} (cf. (2.6) in \cite{DJ}).
We will use $\mu_m$ and $\mu_M$ to stand for the smallest and the biggest principal curvature respectively.
The step following (2.6) in the proof of (2.3) in \cite{DJ} does not seem to be correct to us. In the rest of this subsection we will present a careful study of
the Riccati equations and derive the upper bounds and lower bounds of the principal curvature that hold on AH manifolds.
\\

The first is a general fact for complete Riemmannian manifolds.

\begin{lem} Suppose that $(M^n, \ g)$ is a complete Riemannian manifold and $p\in M^n$ is a fixed point. For any given $t_0>0$, there is a constant $C$ such that
\begin{equation}\label{upper-bound}
\mu_M (q) \leq C
\end{equation}
for all $q\in \Gamma_{t_0}\setminus C_p$, where $\Gamma_{t_0}$ is the geodesic sphere of radius $t_0$ and $C_p$ is the cut loci of $p$ in $(M^n, \ g)$.
\end{lem}

\begin{proof} Suppose otherwise that there is a sequence $q_k\in \Gamma_{t_0}\setminus C_p$ such that
$$
\mu_M(q_k) \to \infty.
$$
Without loss of generality one may assume that $q_k\to q_0\in \Gamma_{t_0}$ and the geodesic $\gamma_k(t)$ that realizes the distance and connects $q_k$ and $p$
converges to a minimizing geodesic $\gamma(t)$ from $q_0$ to $p$. Hence $q_0$ is the first cut locus of $p$ along $\gamma$. Set
$$
\mu = \mu_M(\gamma(\frac 12 t_0)).
$$
Since the set $C_p$ of cut loci of $p$ is closed, there is an open neighborhood $U$ of $\gamma(\frac 12t_0)$ where $t$ is smooth. One may assume that
$$
\mu_M(q)\leq 2\mu
$$
for all $q\in U$ (take a smaller set if necessary). Particularly
$$\gamma_k(\frac 12t_0)\in U \text{ and } \ \mu_M(\gamma_k(\frac 12t_0))\leq 2\mu
$$
for all sufficiently large $k$, which is easily seen to contradict with the fact that $$\mu_M(q_k)  = \mu_M(\gamma_k(t_0))\to\infty$$ in the light of the Riccati equation \eqref{riccati}.
So the proof is complete.
\end{proof}

Consequently we have the sharp upper bound for the principal curvature of geodesic spheres in AH manifolds.

\begin{cor} Suppose that $(X^n, \ g^+)$ is AH and $p\in X^n$ is a fixed point. And suppose that the curvature condition \eqref{cur-decay}
holds on $(X^n, \ g^+)$. Then, for $t_0>0$,  there exists a constant $C$ such that
\begin{equation}\label{sharp-upper-bound}
\mu_M (q) \leq 1 +C(t+1)e^{-2t}
\end{equation}
for all $q\in \Gamma_t\setminus C_p$ and all $t\geq t_0$.
\end{cor}
\begin{proof} From \eqref{riccati-1} one has
$$
\mu_M' + \mu_M^2 \leq 1 + C_0e^{-2t}.
$$
Then one may consider $z = \mu_M -1$ along the minimizing geodesic $\gamma (s)$ from $p$ to $q\in\Gamma_t$
and the equation \eqref{riccati-1} for $z$
$$
z' + 2z \leq C_0e^{-2s} - z^2 \leq C_0e^{-2s},
$$
which easily implies \eqref{sharp-upper-bound}. We remark that the constant $C$ may very in different places.
\end{proof}

Roughly speaking, the reason one has the sharp upper bound \eqref{sharp-upper-bound} is because $1$ is a sink for the equation
\begin{equation}\label{riccati-2}
\mu' + \mu^2 = 1.
\end{equation}
On the other hand, $-1$ is a source for the equation \eqref{riccati-2}, one could not expect a lower bound in general. In fact the best
one can say about the lower bound for the principal curvature is as follows:

\begin{prop}\label{no-lower-b} Suppose that $(X^n, \ g^+)$ is AH and $p\in X^n$ is a fixed point. And suppose that the curvature condition \eqref{cur-decay}
holds on $(X^n, \ g^+)$. Then
$$
\mu_m(q) \geq -\sqrt{1 + C_0e^{-2t}}
$$
for $q\in\Gamma_t$ and $q$ is on a minimizing geodesic ray that runs from $p$ to the infinity without intersecting the set $C_p$
of cut loci of $p$.
\end{prop}
\begin{proof} Suppose that $\gamma$ is a minimizing geodesic that runs from $p$ to the infinity without intersecting the set $C_p$
of cut loci of $p$. Assume otherwise
$$
\mu_m (\gamma(t)) < - \sqrt{1 +C_0e^{-2t}}.
$$
Then, based on \eqref{riccati-1},  it is not hard to show that $\mu_m(\gamma(s))$ goes to $-\infty$ in a finite time after $t$,
which is a contradiction.
\end{proof}

\subsection{Volume estimates where the mean curvature is not big enough} In this section we are concerned with the set
$$
U^\delta_r = \{q\in \Sigma_r\setminus C_p: H(q) < (n-1)( 1- \delta)\}
$$
where $H = \Delta t$. We would like to show that
$$
\int_{U^\delta_r} dv[\tilde g_r] \to 0
$$
as $r\to\infty$.
\\

We first notice from \eqref{angle-r-t} in the proof of Lemma \ref{liptschitz} that
\begin{equation}
\|dv[g^+](\nabla t) - dv[g^+](\nabla r)\|_{g^+} \leq Ce^{-t}
\end{equation}
and
\begin{equation}\label{com-vol-element}
|\frac {2^{n-1}e^{-(n-1)t}dv[g^+](\nabla r)} { 2^{n-1}e^{-(n-1)t}dv[g^+](\nabla t)|_{(\nabla r)^\perp}} - 1|
= |\frac {dv[g^+](\nabla r)} {dv[g^+](\nabla t)|_{(\nabla r)^\perp}} - 1|
\leq Ce^{-2t},
\end{equation}
where
\begin{equation}\label{volume-element}
dv[\tilde g_r] = 2^{n-1}e^{-(n-1)t}dv[g^+](\nabla r)
\end{equation}
and $(\nabla r)^\perp$ stands for the hyperplane perpendicular to $\nabla r$. This gives us the clue how the mean curvature
$H = \Delta t$ is related to the volume element $dv[\tilde g_r]$
on $\Sigma_r$. We now focus on the behavior of the (n-1)-form $dv[g^+](\nabla t)$ along geodesics $\gamma$ from $p$.
\\

To be more clearer we set a parallel orthonormal frame $\{\omega_1,
\omega_2\cdots, \omega_{n-1}\}$ along a geodesic $\gamma$ for the dual of the subspace perpendicular to the geodesic $\gamma$. Then we may write
$$
dv[g^+](\nabla t) = {\mathcal J}\omega_1\cdot\omega_2\cdots\omega_{n-1}.
$$
Now we recall from the first variation formula in Riemannian geometry that
\begin{equation}\label{riccati-4}
\frac {d{\mathcal J}}{dt} = H {\mathcal J}.
\end{equation}
To state the key observation for the volume estimate, for each small $\delta$, we let $t_\delta$ be a fixed large number such that
$$
(n-2)C(t+1)e^{-2t} \leq \frac \delta 4 \text{ and } 1 - C_0e^{-2t}\geq (1 - \frac 12\delta)^2
$$
for all $t\geq t_\delta$, where $C$ is the constant in \eqref{sharp-upper-bound} and $C_0$ is the constant in \eqref{cur-decay}.

\begin{lem} Suppose that $(X^n, \ g^+)$ ($n\geq 4$) is AH of $C^3$ regularity and $p\in X^n$ is a fixed point. And suppose that the curvature condition \eqref{cur-decay} holds.
Let $q\in U^\delta_r$ and $\gamma(s)$ be the minimizing geodesic
from $p$ to $q = \gamma(t)$ for $t> t_\delta$. Then
\begin{equation}\label{small-H}
H(\gamma(s))\leq (n-1)( 1 - \frac \delta{4(n-1)})
\end{equation}
for all $s\in (t_\delta, t)$.
\end{lem}

\begin{proof} Assume otherwise, for some $s\in (t_\delta, t)$ such that
$$
H(\gamma(s)) > (n-1)( 1 - \frac \delta{4(n-1)}).
$$
Using the sharp upper bound \eqref{sharp-upper-bound} we conclude that
\begin{equation}\label{too-large-lb}
\mu_m(\gamma(s)) > 1 - \frac {\delta} 2.
\end{equation}
It is then easily seen that \eqref{too-large-lb} implies
$$
\mu_m(q) = \mu_m(\gamma(t))\geq 1 -\frac 34\delta >1 - \delta.
$$
Because, from \eqref{riccati-1}, one has
$$
\mu_m' + \mu_m^2 \geq 1 - C_0e^{-2t} \geq (1- \frac 12\delta)^2
$$
in $(t_\delta, t)$, which means $\mu_m'$ is positive whenever $\mu_m\in (0, 1 - \frac 12\delta)$.
\end{proof}

Consequently, from \eqref{volume-element}, \eqref{com-vol-element}, and \eqref{riccati-4}, considering $U_r^{\delta}$ as a graph on a subset of $\Gamma_{t_{\delta}}$ induced by exponential map at $p$, we have

\begin{prop} \label{vol-upper} Suppose that $(X^n, \ g^+)$ is an AH manifold of $C^3$ regularity and $p\in X^n$ is a fixed point. And suppose that the curvature condition \eqref{cur-decay} holds. Then
\begin{equation}\label{s-volume-growth}
{\mathcal J}(q) \leq {\mathcal J}(\gamma(t_\delta))e^{(n-1)(1 - \frac \delta{4(n-1)}) (t-t_\delta)}
\end{equation}
for $q\in \Gamma_t\setminus C_p$ with $H < (n-1)(1-\delta)$ and $\gamma$ is the minimizing geodesic from $p$ to $q$. Thus, for each fixed small $\delta >0$,
\begin{equation}\label{s-volume}
\int_{U^\delta_r} dv[\tilde g_r] \to 0
\end{equation}
as $r\to\infty$.
\end{prop}

\subsection{Volume estimates where the mean curvature is very negative} If there were a lower bound for the mean curvature, then \eqref{estimate-0} and \eqref{sharp-upper-bound}
would yield a upper bound the scalar curvature $R[\tilde g_r]$ and one would have by now completed the proof of \eqref{gap-2}. Therefore we need to estimate the lower bound
of $Hdv[g^+](\nabla t)$ at least when the mean curvature $H$ is very negative. First of all one easily derives from the first variational formula \eqref{riccati-4} that, along a geodesic
$\gamma(s)$,
\begin{equation}\label{riccati-5}
\frac{d}{ds} (Hdv[g^+](\nabla t)) = (H' + H^2)dv[g^+](\nabla t).
\end{equation}

The key observation for this subsection is the following more detailed statement of Proposition \ref{no-lower-b} about the Riccati equation:

\begin{lem}\label{uniform-lb} Suppose that $(X^n, \ g^+)$ is AH and $p\in X^n$ is a fixed point. And suppose that the curvature condition \eqref{cur-decay} holds. Then, for any $t_0>0$,
 there is a constant $C>0$,
for $q\in \Gamma_t\setminus C_p$ and $\gamma$ being the minimizing geodesic from $p$ to $q$, one always has
\begin{equation}\label{cur-t-1}
\mu_m(\gamma(t-1)) \geq - C,
\end{equation}
provided that $t> t_0 +1$.
\end{lem}

\begin{proof} To start, one considers the Riccati equation \eqref{riccati-1} along the geodesic $\gamma$ when $\mu (s_0) < -\sqrt{1 + C_0e^{-2s_0}}$ with $s_0=t-1>t_0$. Recall
$$
\mu' + \mu^2 \leq 1 + C_0e^{-2s}\leq 1 + C_0e^{-2s_0}
$$
for $s > s_0$, which implies
$$
\mu(s) \leq - a \frac {1 + \chi e^{2a(s -s_0)}}{1 - \chi e^{2a(s-s_0)}}
$$
where $a = \sqrt{1+C_0e^{-2s_0}}$ and $\chi =  - \frac {a + \mu(s_0)}{a - \mu(s_0)} \in(0, 1)$. Therefore it is clear that $\mu$ reaches $-\infty$ at
$$
s_1 = t-1  - \frac{1}{2a} \log\chi
$$
and
$$
\lim_{\mu(t-1)\to -a} \chi= 0 \text{ and } \lim_{\mu(t-1)\to -\infty} \chi = 1.
$$
Thus
$$
-\frac{1}{2a}\\log \chi < 1
$$
when $\mu(t-1)$ is sufficient negative, which implies, there is a constant $C=C(t_0)$ such that
$$
\mu_m(\gamma(t-1)) > -C
$$
for $\gamma(t)\in \Gamma_t\setminus C_p$ and $t>t_0+1$.
\end{proof}

Now we are ready to deal with the volume estimate at places the mean curvature is very negative.

\begin{prop} \label{H-very-neg} Suppose that $(X^n, \ g^+)$ is AH and $p\in X^n$ is a fixed point. And suppose that the curvature condition \eqref{cur-decay} holds. Then there is a constant $C>0$
such that, for $q\in \Gamma_t\setminus C_p$ with $H(q)\leq - 2(n-1)$ and $\gamma$ being the minimizing geodesic from $p$ to $q$,
\begin{equation}\label{H-lower-b}
H{\mathcal J} (\gamma(t)) \geq -C e^{(n-1)(1 - \frac\delta{4(n-1)})t},
\end{equation}
provided that $t> t_\delta+1$.
\end{prop}

\begin{proof} Using Riccati equation and \eqref{riccati-5}, we end at
$$
\frac{d}{ds} (H{\mathcal J}) = (-\text{Ric}[g^+] (\nabla t, \nabla t) - |\nabla^2 t|^2 + H^2){\mathcal J}
$$
for $s\in (t-1, t)\subset(t_\delta, t)$. Hence we are looking for lower bound for $(\Delta t)^2 - |\nabla^2 t|^2$. Let  $$\{\mu_i: i=1, \cdots, n-1\}$$ be the eigenvalues of
$\nabla^2 t$.  We may calculate
$$
(\Delta t)^2 - |\nabla^2 t|^2 = \displaystyle\sum_{i\neq j} \mu_i\mu_j \geq \sum_{\mu_i\cdot\mu_j <0}\mu_i\mu_j\geq 2(n-1)^2\mu_m\geq 2(n-1)^2H - C
$$
for some constant $C>0$, where the last inequality uses the sharp upper bound \eqref{sharp-upper-bound}. Notice also that $H\leq -2(n-1)$ yields that
$\mu_m \leq -2$, which guarantees that there is always some negative eigenvalues. Therefore by $(\ref{cur-decay})$, we obtain
\begin{equation}\label{good-ode}
\frac{d}{ds} (H{\mathcal J}) \geq 2(n-1)^2H{\mathcal J} - C{\mathcal J}
\end{equation}
for some constant $C>0$, which can be rewritten as follows:
$$
\frac d{ds}(e^{-2(n-1)^2s}H{\mathcal J}) \geq - C e^{-2(n-1)^2s}{\mathcal J}
$$
for $s\in(t-1, t)$ Thus
\begin{equation}\label{integrated-1}
H{\mathcal J} (\gamma(t)) \geq  e^{2(n-1)^2}H{\mathcal J}(\gamma(t-1)) -Ce^{2(n-1)^2t} \int_{t-1}^t e^{-2(n-1)^2s} {\mathcal J}(\gamma(s))ds.
\end{equation}
Now, in the light of Lemma \ref{uniform-lb}, \eqref{H-lower-b} is proven if \eqref{s-volume-growth} in Proposition \ref{vol-upper} is applicable to all points $\gamma(s): s\in [t-1, t]$.
To see \eqref{s-volume-growth} holds at each point $\gamma(s): s\in [t-1, t]$, one uses that $\mu_m (\gamma(t)) < -2$ and the Riccati equation \eqref{riccati-1} for $t-1 > t_\delta$.
Because one gets that $\mu_m(\gamma(s)) < -1+C_0e^{-2t_{\delta}}$ and hence $H(\gamma(s))\leq (n-1)(1-\delta)$ for all $s\in[t-1, t]$, at least when $\delta$ is small enough.
\end{proof}

Combining Proposition \ref{vol-upper} and Proposition \ref{H-very-neg} we thus have obtained

\begin{thm} \label{step-two} Suppose that $(X^n, \ g^+)$ is AH of $C^3$ regularity and $p\in X^n$ is a fixed point. And suppose that \eqref{scalar-cur-asym} holds.
Then for almost all large $r>0$ so that $\mathcal {H}^{n-2}(\gamma_r^{QL})=0$, it holds that
\begin{equation}\label{total-sclar-cur}
\int_{\Sigma_r\setminus(B^r_\epsilon\cup N_p)}(Rdv)[\tilde g_r] \leq (n-1)(n-2)\int_{\Sigma_r}dv[\tilde g_r] + o_r(1),
\end{equation}
where $o_r(1)$ is independent of $\epsilon>0$ and $o_r(1)\to 0$ as $r\to\infty$.
\end{thm}
\begin{proof}
At any point $q\in \Sigma_r\setminus C_p$, by $(\ref{com-vol-element})$,
\begin{align*}
 |\frac {Hdv[g^+](\nabla r)} {Hdv[g^+](\nabla t)|_{(\nabla r)^\perp}} - 1| \leq Ce^{-2t}.
\end{align*}
Since $|t-r|$ is uniformly controlled, for $q\in \Sigma_r\setminus C_p$ so that $H\leq -2(n-1)$ we get the decay of $H\,dv[g^+](\nabla r)$ by Proposition \ref{H-very-neg}.
\\

For any $\delta>0$, for $r>0$ large as in the statement of the theorem, considering $\Sigma_r$ as a graph of $\Gamma_{t_{\delta}}$ induced by the exponential map at $p$, we have the splitting of the integration,
\begin{align*}
&\int_{\Sigma_r\setminus(B_{\epsilon}^r\bigcup N_p)} (R dV)[\tilde{g}_r]\\
&=\int_{\{q\in \Sigma_r\setminus(B_{\epsilon}^r\bigcup N_p):\,H(q)>(n-1)(1-\delta)\}} (R dV)[\tilde{g}_r]+\int_{\{q\in \Sigma_r\setminus(B_{\epsilon}^r\bigcup N_p): \,-2(n-1)<H\leq (n-1)(1-\delta)\}} (R dV)[\tilde{g}_r]\\
&+\int_{\{q \in\Sigma_r\setminus(B_{\epsilon}^r\bigcup N_p):\,H\leq -2(n-1)\}} (R dV)[\tilde{g}_r].
\end{align*}
Note that $Ddt(v, v)\geq \mu_m \geq H-C$. By Proposition 6.5 and Proposition 6.7 combining with $(6.11)$ and $(6.12)$ and arbitrary choice of $\delta>0$, we completes the proof of $(6.23)$.
\end{proof}


 \end{document}